\documentclass[a4paper,10pt,leqno]{amsart}
\title{On the algebraic $K$-theory of Hecke algebras}
       \author{Bartels, A.}
       \address{Westf\"alische Wilhelms-Universit\"at M\"unster\\
               Mathematicians Institut\\
               Einsteinium.~62,
               D-48149 M\"unster, Germany}
        \email{bartelsa@math.uni-muenster.de}
        \urladdr{http://www.math.uni-muenster.de/u/bartelsa}
        \author{L\"uck, W.}
        \address{Mathematicians Institut der Universit\"at Bonn\\
                Endenicher Allee 60\\
                53115 Bonn, Germany}
         \email{wolfgang.lueck@him.uni-bonn.de}
          \urladdr{http://www.him.uni-bonn.de/lueck}
         \date{April, 2022}
         \keywords{Hecke algebras , algebraic $K$-theory, Wang sequence.}
          \makeatletter
     \@namedef{subjclassname@2020}{%
  \textup{2020} Mathematics Subject Classification}
\makeatother
\subjclass[2020]{20C08,19D35,19D50}

\usepackage[utf8]{inputenc}
  \usepackage{hyperref}
  \usepackage{comment}
  \usepackage{calc}
  \usepackage{enumerate,amssymb,bm}
  \usepackage[arrow,curve,matrix,tips,2cell]{xy}
    \SelectTips{eu}{10} \UseTips
    \UseAllTwocells
  \usepackage{tikz}
  \usetikzlibrary{decorations.pathreplacing}
  \usepackage{color}
  \usepackage{scalerel}
  \usepackage{braket}
  \usepackage{mathtools}
  \usepackage{bbm}
  \usepackage{enumitem}
  \usepackage{float}
  \DeclareMathAlphabet{\matheurm}{U}{eur}{m}{n}

\DeclareMathAlphabet{\matheurm}{U}{eur}{m}{n}

\newcommand{\MODcat}[1]{#1\text{-}\matheurm{MOD}}

\newcommand{\SubGF}[2]{\matheurm{Sub}_{#2}(#1)}

\DeclareMathOperator{\aut}{aut}

\DeclareMathOperator{\cent}{cent}

\DeclareMathOperator{\cok}{cok}
\DeclareMathOperator{\colim}{colim}

\DeclareMathOperator{\FP}{FP}
\DeclareMathOperator{\fgf}{fgf}
\DeclareMathOperator{\fgp}{fgp}

\DeclareMathOperator{\hocolim}{hocolim}

\DeclareMathOperator{\id}{id}
\DeclareMathOperator{\im}{im}

\DeclareMathOperator{\Idem}{Idem}

\DeclareMathOperator{\mor}{mor}

\DeclareMathOperator{\pr}{pr}

\DeclareMathOperator{\supp}{supp}

\newcommand{\COP}{{\calc\hspace{-1pt}\mathrm{op}}}

\newcommand{\CVCYC}{{\calc\hspace{-1pt}\mathrm{vcy}}}

\newcommand{\FIN}{{{\mathcal F}\mathrm{in}}}

  \newcommand{\IC}{\mathbb{C}}

  \newcommand{\IN}{\mathbb{N}}

  \newcommand{\IQ}{\mathbb{Q}}
  \newcommand{\IR}{\mathbb{R}}

  \newcommand{\IZ}{\mathbb{Z}}

  \newcommand{\cala}{\mathcal{A}}
  \newcommand{\calb}{\mathcal{B}}
  \newcommand{\calc}{\mathcal{C}}

  \newcommand{\calh}{\mathcal{H}}

  \newcommand{\calk}{\mathcal{K}}

  \newcommand{\calp}{\mathcal{P}}

  \newcommand{\bfa}{\mathbf{a}}

  \newcommand{\bfT}{\mathbf{T}}

 \newcommand{\bfKinfty}{\mathbf{K}^{\infty}}

 \newcommand{\bfNKinfty}{\mathbf{N}\mathbf{K}^{\infty}}

\newcommand{\EGF}[2]{E_{#2}(#1)}

\newcommand{\squarematrix}[4]
{\left( \begin{array}{cc} #1 & #2 \\ #3 &
#4
\end{array} \right)
}

\newcounter{commentcounter}

\theoremstyle{plain}
\newtheorem{theorem}{Theorem}[section]

\newtheorem{lemma}[theorem]{Lemma}

\newtheorem{condition}[theorem]{Condition}
\newtheorem*{theorem*}{Theorem}
\newtheorem*{theoremA*}{Theorem A}
\newtheorem*{theoremB*}{Theorem B}

\theoremstyle{definition}
\newtheorem{definition}[theorem]{Definition}

\newtheorem{remark}[theorem]{Remark}

\newtheorem*{definition*}{Definition}

\theoremstyle{remark}

\makeatletter\let\c@equation=\c@theorem\makeatother

\theoremstyle{definition}

\newcounter{othercommentcounter}

\newcommand{\version}[1]              
{\begin{center} last edited on #1\\
last compiled on \today\\
name of tex-file: \jobname
\end{center}
}

\begin{document}

\begin{abstract}
  Consider a totally disconnected group $G$, which is covirtually cyclic, i.e., contains a
  normal compact open subgroup $L$ such that $G/L$ is infinite cyclic. We establish a Wang
  sequence, which computes the algebraic $K$-groups of the Hecke algebra of $G$ in terms
  of the one of $L$, and show that all negative $K$-groups vanish.  This confirms the
  $K$-theoretic Farrell-Jones Conjecture for the Hecke algebra of $G$ in this special
  case.  Our ultimate long term goal is to prove it for any closed subgroup of any
  reductive $p$-adic group. The results of this paper will play a role in the final proof.
\end{abstract}

     \maketitle

     \newlength{\origlabelwidth}\setlength\origlabelwidth\labelwidth


     \typeout{------------------- Introduction -----------------}
     \section{Introduction}\label{sec:introduction}

     Let $G$ be a td-group, i.e., a locally compact second countable totally disconnected
     topological Hausdorff group.  Our ultimate goal is to compute the algebraic
     $K$-groups and in particular the projective class group of the Hecke algebra
     $\calh(G)$ of $G$, which is defined in terms of locally constant functions with
     compact support from $G$ to the real or complex numbers and the convolution product.
     We want to show that the canonical map
     \begin{equation}
       \colim_{K \in \SubGF{G}{\COP}} K_0(\calh(K)) \xrightarrow{\cong} K_0(\calh(G))
       \label{colim_K_K_0(calh(K)_is_K_0(calh(G))}
     \end{equation}
     is bijective. Here $\SubGF{G}{\COP}$ is the following category. Objects are compact
     open subgroups $K$ of $G$, a morphism $f \colon K \to K'$ is a group homomorphism,
     for which there exists $g \in G$ satisfying $f(k) = gkg^{-1}$ for all $k \in K$, and
     we identify two such group homomorphisms $f \colon K \to K'$ and
     $f' \colon K \to K'$, if they differ by an inner automorphism of $K'$.  In particular
     the obvious map
     \begin{equation}
       \bigoplus_{K} K_0(\calh(K)) \to K_0(\calh(G))
       \label{surjectivity_of_induction_on_K_0}
     \end{equation}
     is surjective, where $K$ runs through the compact open subgroups of $G$.

     Dat~\cite[Theorem~1.6 and Corollary~4.22]{Dat(2000)} showed following ideas of
     Bernstein that the map~\eqref{surjectivity_of_induction_on_K_0} is rationally
     surjective for a reductive $p$-adic group $G$. He used for the proof the
     Hattori-Stallings rank and input from the representation theory of reductive
     $p$-adic groups.  Dat also asked the question, whether the
     map~\eqref{surjectivity_of_induction_on_K_0} is surjective without rationalizing, see
     the sentence after~\cite[Proposition~1.10]{Dat(2003)} and the formulation of the
     weaker conjecture~\cite[Conjecture~1.11]{Dat(2003)}.

     The projective class group $K_0(\calh(G))$ is interesting for the study of smooth
     $G$-representations, since  every finitely generated smooth
     $G$-representation has a finite projective resolution and hence define elements in it, see for
     instance~\cite[Theorem~29 on page~97 and Proposition~32 on page~60]{Bernstein(1992)},%
~\cite{Schneider-Stuhler(1991)},~\cite{Schneider-Stuhler(1993)},~\cite{Schneider-Stuhler(1997)},~\cite{Vigneras(1990)}.

     If $G$ is discrete, the family $\COP$ of compact open subgroups reduces to the family
     $\FIN$ of finite subgroups of $G$ and the bijectivity of the
     map~\eqref{colim_K_K_0(calh(K)_is_K_0(calh(G))} reduces to the bijectivity of the
     canonical map
     \begin{equation}
       \colim_{F \in \SubGF{G}{\FIN}} K_0(\IC F) \xrightarrow{\cong} K_0(\IC G),
       \label{colim_K_K_0(CF)_is_K_0(CG)}
     \end{equation}
     which follows from the $K$-theoretic Farrell-Jones Conjecture for $\IC G$.

     Our ultimate and long term goal is to the prove the version of the
     \emph{$K$-theoretic Farrell-Jones Conjecture for the Hecke algebra of td-groups} for
     any closed subgroup $G$ of any reductive $p$-adic group. It predicts the bijectivity
     of the assembly map
     \begin{equation}
       H_n^G(\EGF{G}{\COP};\bfKinfty_{\calh}) \xrightarrow{\cong} H_n^G(G/G;\bfKinfty_{\calh}) = K_n(\calh(G))
       \label{FJ_assembly_standard_Hecke}
     \end{equation}
     for every $n \in \IZ$. Here the source is a smooth $G$-homology theory, which digests
     smooth $G$-$CW$-complexes and satisfies
     $H^G_n(G/H;\bfKinfty_{\calh}) = K_n(\calh(H))$ for open subgroups $H \subseteq G$,
     and the smooth $G$-$CW$-complex $\EGF{G}{\COP}$ is a model for the classifying space
     of the family of compact open subgroups, or, equivalently the classifying space for
     smooth proper $G$-actions in the realm of $G$-$CW$-complexes. This map will be
     constructed in~\cite{Bartels-Lueck(2020foundations)}, where a formulation of the
     $K$-theoretic Farrell-Jones Conjecture is given for Hecke categories, which
     generalize the notion of a Hecke algebra.

     We will not prove the $K$-theoretic Farrell-Jones Conjecture for Hecke categories in
     this paper. At least we present a direct proof of it in the special case that $G$ is
     covirtually infinite cyclic, i.e., $G$ contains a normal compact open subgroup $L$
     such that the quotient $G/L$ is the discrete group $\IZ$. Then the conjecture boils
     down to Theorem~\ref{the:Wang_sequence} which 
     says that there is a Wang sequence, infinite to the left,
     \begin{multline*}
       \cdots \xrightarrow{K_2(i)} K_2(\calh(G)) \xrightarrow{\partial_2} K_1(\calh(L))
       \xrightarrow{\id - K_1(\phi)} K_1(\calh(L))
       \\
       \xrightarrow{K_1(i)} K_1(\calh(G)) \xrightarrow{\partial_1} K_0(\calh(L))
       \xrightarrow{\id - K_0(\phi)} K_0(\calh(L))
       \\
       \xrightarrow{K_0(i)} K_0(\calh(G)) \to 0,
     \end{multline*}
     where $\phi \colon L \to L$ is the automorphism given by conjugation with some
     preimage of the generator of the infinite cyclic group $G/L$ under the projection
     $G \to G/L$ and $i \colon L \to G$ is the inclusion, and that we have
     \[
       K_n(\calh(G)) = 0 \quad \text{for} \;n \le -1.
     \]
     So in this paper we can confirm the Farrell-Jones Conjecture for covirtually infinite
     cyclic td-groups.  One may say that this paper plays the same role for the
     Farrell-Jones Conjecture for Hecke algebras as the papers by
     Farrell-Hsiang~\cite{Farrell-Hsiang(1970)} and 
     Pimsner-Voiculescu~\cite{Pimsner-Voiculescu(1982)} did for the Farrell-Jones Conjecture for discrete groups
     and the Baum-Connes  Conjecture. To our knowledge this paper presents
     the first instance of a version of the Farrell-Jones Conjecture for non-discrete groups.

     One application of this paper will be that the bijectivity
     of~\eqref{FJ_assembly_standard_Hecke} implies the bijectivity
     of~\eqref{colim_K_K_0(calh(K)_is_K_0(calh(G))}. Moreover, Theorem~\ref{the:regular_coherence_of_the_Hecke_algebra}
     and Theorem~\ref{the:Q'_to_Q-faithfully_flat}  will be key ingredients
     in the part of the forthcoming proof of the Farrell-Jones Conjecture, where we will
     reduce the family $\CVCYC$ of (not necessarily open) covirtually cyclic subgroups to
     the family $\COP$.

     We mention that we will look at more complicated Hecke algebras than the standard
     ones. We will allow other rings than $\IR$ or $\IC$. Moreover, we take a $G$-action on $R$ by
     ring automorphisms and a normal character, which is an obvious generalization of a
     central character, into account. In the sequel papers we will replace the
     Hecke algebras by the more general notion of a Hecke category, since allowing more general coefficients
     will ensure the desirable inheritance to closed subgroups of the Farrell-Jones Conjecture.
      This is interesting in the case of reductive $p$-adic groups, since important
     subgroups such as the Borel subgroup are in general not open.

     One ingredient for the main results of this paper is the Bass-Heller-Swan
     decompositions for additive categories and the presentation of criteria for the
     vanishing of the Nil-term,
     see~Section~\ref{sec:A_review_of_the_Bass-Heller-Swan_decomposition_for_unital_additive_categories},~%
     and~\cite{Bartels-Lueck(2020additive),Lueck-Steimle(2016BHS)}. The second is the
     analysis of the filtration of the Hecke algebra of a compact td-groups in terms of
     approximate units,  see
     Section~\ref{sec:Hecke_algebras_over_compact_td-groups_and_crossed_product_rings}.


     \subsection{Conventions and notations}\label{subsec:Conventions_and_notations}

     \begin{itemize}

     \item A td-group is a locally compact second countable totally disconnected
       topological Hausdorff group;

     \item A subgroup is always assumed to be closed;

     \item A group homomorphism has closed image and is an identification onto it;

     \item We denote by $R$ an associative ring, which is not necessarily commutative and
       not necessarily has a unit.  If a ring has a unit, it is called a \emph{unital
         ring}. In almost all cases we will require for a unital ring $R$ that
       $\IQ \subseteq R$ holds, i.e., for every integer $n \ge 1$ the element
       $n \cdot 1 = 1 + 1+ \cdots + 1$ has a multiplicative inverse in $R$;

     \item In a ring the unit is denoted by $1$. In a group the unit is denoted by $e$;

     \item For an epimorphism $p \colon S \to S'$ of sets, a \emph{transversal} $T$ is a
       subset $T \subseteq S$ such that the restriction of $p$ to $T$ yields a bijection
       $p|_T \colon T \xrightarrow{\cong} S'$. If $S$ is a group, we always assume that
       the unit is in $T$;

     \end{itemize}


     \subsection{Acknowledgments}\label{subsec:Acknowledgements}

      The paper is funded by the ERC Advanced Grant “KL2MG-interactions” (no. 662400) of
     the second author granted by the European Research Council, by the Deutsche
     Forschungsgemeinschaft (DFG, German Research Foundation) under Germany’s Excellence
     Strategy – GZ 2047/1, Projekt-ID 390685813, Hausdorff Center for Mathematics at
     Bonn, and by the Deutsche Forschungsgemeinschaft (DFG, German Research Foundation) –
     Project-ID 427320536 – SFB 1442, as well as under Germany’s Excellence Strategy EXC
     2044 390685587, Mathematics Münster: Dynamics–Geometry–Structure.


     \subsection{Laudatio for Dr. Catriona Byrne by the second
       author}\label{subsec:Laudatio}

     When I was invited to contribute to the Festschrift in honor of Springer’s Editorial
     Director Dr. Catriona Byrne, I was flattered and did not hesitate to accept the
     invitation.  I know her, since I submitted my first book with the title
     ``Transformation groups and algebraic $K$-theory'' to the \emph{Lecture Notes in
       Mathematics} in 1989. She has always been a very reliable and competent partner for
     all of my book projects.  She is one of the few persons working for a publishing
     house, which do manage to keep the right balance between the mercantile aspects and
     the interest of the mathematical community. One always has the impression that she
     does care about the contents and the authors of any submission. In particular we did
     get close, when Andrew Ranicki and I and Springer had a very hard time to deal with
     the problems caused by the managing editor of the journal \emph{$K$-theory} in 2007,
     see ``Pers\"onliches Protokoll zur Zeitschrift K-Theory'' in Mitteilungen der
     Deutschen Mathematiker-Vereinigung, Band~15 Heft~3, 2007.

     I wish Catriona all the best for the many years to come.


     \typeout{------------------- Section 2:  Hecke algebras -----------------}

     \section{Hecke algebras}\label{sec:Hecke_algebras}

     In this section we slightly generalize the notions of a Hecke algebra by implementing a normal character.


     \subsection{Normal characters}\label{subsec:normal_characters}

     Let $R$ be a (not necessarily commutative) associative unital ring with
     $\IQ \subseteq R$.  Let $G$ be a td-group with a normal (not necessarily open or
     central) subgroup $N \subseteq G$. Put $Q = G/N$.  Then we obtain an extension of
     td-groups $1 \to N \to G \xrightarrow{\pr} Q \to 1$.

     Consider a group homomorphism $\rho \colon G \to \aut(R)$, where $\aut(R)$ is the
     group of automorphism of the unital ring $R$.  We will assume  throughout the paper
     that the kernel of $\rho$ is open, in other words, $G$ acts smoothly on $R$.

     We write $gr = \rho(g)(r)$ for
     $g \in G$ and $r \in R$.  With this notation we get $er = r$, $g1 = 1$,
     $(g_1g_2)r = g_1(g_2r)$, $g(r_1r_2) = (gr_1)(gr_2)$ and $g(r_1 + r_2)= gr_1 + gr_2$
     for $g,g_1,g_2 \in G$, $r,r_1,r_2 \in R$, and the units $e \in G$ and $1 \in R$.

A \emph{normal character} is a locally constant group homomorphism
\[
  \omega\colon N \to \cent(R)^{\times}
\]
to the multiplicative group of central units of $R$ satisfying
\begin{eqnarray}
  \omega(gng^{-1}) & = & \omega(n)
                         \label{omega_and_conjugation}
\end{eqnarray}
for all $n \in N$ and $g \in G$.  Note that $\ker(\omega)$ is an open subgroup of $N$ and
a normal subgroup of $G$.  We will need the following compatibility condition between the
normal character and the $G$-action $\rho$ on $R$, namely
for $n \in N$, $g \in G$, and $r \in R$
\begin{eqnarray}
  g \omega(n) & = & \omega(n);
                    \label{rho_and_omega}
  \\
  n \cdot r & = & r.
                  \label{N_and_rho}
\end{eqnarray}


\subsection{The construction of the Hecke algebra}%
\label{subsec:The_construction_of_the_Hecke_algebra}
Let $\mu$ be a \emph{$\IQ$-valued Haar measure on $Q$}, i.e., a Haar measure $\mu$ on $G$
such that for any compact open subgroup $K \subseteq Q$ we have $\mu(K) \in \IQ^{>
  0}$. Given any Haar measure $\mu$ on $G$, we can normalize it to a $\IQ$-valued Haar
measure by choosing a compact open subgroup $L_0 \subseteq G$ and defining
$\mu' = \frac{1}{\mu(L_0)} \cdot \mu$.

An element $s$ in the \emph{Hecke algebra} $\calh(G;R,\rho,\omega)_{\mu}$ is
given by a map $s \colon G \to R$ with the following properties

\begin{itemize}

\item The map  $s \colon G \to R$ is locally constant;
  
\item The image of its support $\supp(s) := \{g \in G \mid s(g) \not= 0\}\subseteq G$
  under $\pr \colon G \to Q$ is a compact subset of $Q$;

\item For $n \in N$ and $g \in G$ we have
  \begin{eqnarray}
    s(ng)  & =  &  \omega(n) \cdot s(g);
                  \label{condition_s_and_omega:left}
    \\
    s(gn)  & =  & s(g) \cdot \omega(n).
                  \label{condition_s_and_omega:right} 
  \end{eqnarray}             
  
\end{itemize}

\begin{definition}\label{def:admissible}
  Let $P_{\rho,\omega}$ the subset of compact open subgroups $K \subseteq G$ satisfying
  \begin{eqnarray}
    kr & = & r \quad \text{for} \; k \in K, r \in R.
   \label{calt_hecke_algebra:tau_and_multiplication_with_r} 
    \\
    \omega(n) & = & 1 \quad \text{for} \; n \in N \cap K;
    \label{tau_and_omega}
  \end{eqnarray}
  We abbreviate $P = P_{\rho,\omega}$ if  $\rho$ and $\omega$ are clear from the context.
  
  We call an element $K \in P$ \emph{admissible}  for $s \colon G \to R$,
  if  for all $g \in G$ and   $k \in K$ we have
  \begin{eqnarray}
    s(kg)  & = &  s(g);
                 \label{condition_s_Hecke_algebra:left}
    \\
    s(gk) 
           & = & 
                 s(g).
                 \label{condition_s_Hecke_algebra:right}
  \end{eqnarray}
  \end{definition}

Note that the existence of an admissible element $K \in P$ is equivalent to the condition
that $s$ is locally constant, since we assume that $s$ has compact support.  Moreover,
for $K \in P$, which is admissible for $s$, every open subgroup $K' \subseteq K$ is also admissible.

\begin{remark}[Redundancy]\label{rem:redundance}
  Note that  condition~\eqref{condition_s_and_omega:right} follows from
  conditions~\eqref{omega_and_conjugation} and~\eqref{condition_s_and_omega:left} by the
  following calculation
  \[
    s(gn) = s(gng^{-1}g) \stackrel{\eqref{condition_s_and_omega:left}}{=} \omega(gng^{-1})
    \cdot s(g) \stackrel{~\eqref{omega_and_conjugation}}{=} \omega(n) \cdot s(g)
    \stackrel{\omega(n) \in \cent(R)}{=} s(g) \cdot \omega(n).
  \]
  Analogously condition~\eqref{condition_s_and_omega:left} follows from
  conditions~\eqref{omega_and_conjugation} and~\eqref{condition_s_and_omega:right}.
\end{remark}

The sum of two elements $s,s'$ in $\calh(G;R,\rho,\omega)_{\mu}$ is defined by
\begin{eqnarray}
  (s+s')(g) := s(g) + s'(g) \; \text{for} \; g \in G.
  \label{sum_IN_calh)}
\end{eqnarray}

Consider $K \in P$ which is admissible for $s$ and  admissible for
$s'$, and a transversal $T$ for the projection $p \colon G \to G/N\!K$, where $N\!K$ is
the subgroup of $G$ given by $\{nk \mid n\in N, k \in K\}$.  Define the product
$s \cdot s'$ by
\begin{eqnarray}
  (s \cdot s')(g) 
  & := &
         \mu(\pr(K)) \cdot \sum_{g' \in T} s(gg') \cdot gg's'(g'^{-1}).
         \label{calt_hecke_algebra:product}
\end{eqnarray}

Note that $K$ may depend on $s$, but not on $g$, whereas $T$
can depend on both $s$ and $g$. The independence of the transversal follows from the following computation for
$g \in G$, $g' \in G'$, $n \in N$ and $k \in K$
\begin{multline*}
  s(g(g'nk)) \cdot g(g'nk)s'((g'nk)^{-1})
  =
  s((gg'n)k) \cdot (gg'n)ks'(k^{-1}n^{-1}g'^{-1})
  \\
  \stackrel{\eqref{condition_s_Hecke_algebra:left},~\eqref{condition_s_Hecke_algebra:right}}{=}
  s(gg'n) \cdot (gg'n)ks'(n^{-1}g'^{-1})
  \stackrel{\eqref{calt_hecke_algebra:tau_and_multiplication_with_r} }{=}
  s(gg'n) \cdot gg'ns'(n^{-1}g'^{-1})
  \\
  \stackrel{\eqref{condition_s_and_omega:left},~\eqref{condition_s_and_omega:right}}{=}
  s(gg') \cdot \omega(n) \cdot gg'n(\omega(n^{-1}) \cdot s'(g'^{-1}))
  =
  s(gg') \cdot \omega(n) \cdot gg'n\omega(n^{-1}) \cdot gg'ns'(g'^{-1})
  \\
  \stackrel{\eqref{rho_and_omega},~\eqref{N_and_rho}}{=}
  s(gg') \cdot \omega(n) \cdot \omega(n^{-1}) \cdot gg's'(g'^{-1})
  =
  s(gg') \cdot \omega(n \cdot n^{-1}) \cdot gg's'(g'^{-1})
  \\
  = s(gg') \cdot \omega(e) \cdot gg's'(g'^{-1}) = s(gg') \cdot  gg's'(g'^{-1}).
\end{multline*}

We leave the elementary proof to the reader that the definition of the
product~\eqref{calt_hecke_algebra:product} is independent of the choice of $K$ and that we
do get the structure of a (non-unital) ring on $\calh(G;R,\rho,\omega)_{\mu}$.  A more
general setting including all proofs will be presented in details
in~\cite{Bartels-Lueck(2020foundations)}.  Moreover, one easily checks

\begin{lemma}\label{lem:condition_compatible_with_product}
  Consider two elements $s,s' \in \calh(G;R,\rho,\omega)_{\mu}$ and compact open subgroups
  $K,K'$ of $G$.  Suppose that $K$ admissible for $s$ and $K'$ is admissible for $s'$.

  Then $K \cap K'$ is admissible for the product $s' \cdot s$.
\end{lemma}


\subsection{Functoriality in $Q$}\label{subsec:Functoriality_in_Q}

Let $G$, $N$, $Q$, $R$, $\rho$, $\omega$, and $\mu$ be as in
Subsection~\ref{subsec:normal_characters}.  In particular we can consider the Hecke
algebra $\calh(G;R,\rho,\omega)_{\mu}$ see
Subsection~\ref{subsec:The_construction_of_the_Hecke_algebra}.

Consider a (not necessarily injective or surjective) open group homomorphism
$\phi \colon G' \to G$ of td-groups. Let $N' \subseteq G'$ be a normal subgroup satisfying
\begin{equation}
  \phi(N') =  N.
  \label{phi(N')_is_N}
\end{equation}
Denote by $\pr' \colon G' \to Q' := G'/N'$ the projection. Let
$\overline{\phi} \colon Q' \to Q$ be the open group homomorphism induced by $\phi$.
Define a group homomorphism $\rho' \colon G' \to \aut(R)$ and a normal character
$\omega' \colon N' \to \cent(R)^{\times}$ by
\begin{eqnarray}
  \rho'
  & = &
        \rho \circ \phi;
        \label{rho_is_rho'_circ_phi}
  \\
  \omega'(n')
  & = &
        \omega(\phi(n')) \quad \text{for}\; n' \in N'.
        \label{omega'(n')_is_omega(beta(n')}
\end{eqnarray}
Choose a $\IQ$-valued Haar measure on $\mu'$ on $Q'$.  Then we can consider the Hecke
algebra $\calh(G';R,\rho',\omega')_{\mu'}$. Next we want to construct a homomorphism
of rings
\begin{equation}
  \phi_* \colon \calh(G';R,\rho',\omega')_{\mu'}  \to \calh(G;R,\rho,\omega)_{\mu}.
  \label{phi_ast_coming_from_open_phi}
\end{equation}
Consider an element $s' \colon G' \to R$ in
$\calh(G';R,\rho',\omega')_{\mu'}$. Choose $K' \in P_{\rho',\omega'}$,  which is
admissible for $s'$. Then $\phi(K') \in P_{\rho,\omega}$.
Fix $g \in G$. Consider $g' \in \phi^{-1}(g\phi(N'K'))$.  Then $\phi(g')^{-1}g$ belongs to
$\phi(N'K')$. Choose $n' \in N'$ and $k' \in K'$ with $\phi(g'n'k') = g$. Put
\begin{eqnarray}
  \widetilde{s'}(g',g)
  & := &
         s'(g')  \cdot \omega(\phi(n'))  \in R.
         \label{widetilde(s)(g',g)}
\end{eqnarray}
One easily checks  that this definition independent of the
choice of $n' \in N'$ and $k' \in K'$.  Obviously we have
$\widetilde{s'}(g',\phi(g')) = s'(g')$ for $g' \in G'$.  Choose a transversal $T'$ of the
projection $G' \to G'/N'K'$, which is allowed to depend on $s'$. Put
$T'(g) = T' \cap \phi^{-1}(g\phi(N'K'))$. Then we  define
\begin{eqnarray}
  \phi_*(s') (g)
  & = & 
        \frac{\mu'(\pr'(K'))}{\mu(\pr(\phi(K')))} \cdot \sum_{g' \in T'(g)} \widetilde{s'}(g',g).
        \label{phi_ast(s')(g)}
\end{eqnarray}
This is a well-defined element in $\calh(G;R,\rho,\omega)_{\mu}$, which is
independent of the choice of $T$ and $K'$. One easily checks

\begin{lemma}\label{lem:properties_of_phi_ast(s')}
  \
  \begin{enumerate}

  \item\label{lem:properties_of_phi_ast(s'):support} We have
    $\supp(\phi_*(s')) \subseteq \phi(\supp(s'));$

  \item\label{lem:properties_of_phi_ast(s'):phi(K'):admissible} If $K' \in P'_{\rho',\omega'}$ is
    admissible for $s'$, then $\phi(K')$ admissible for $\phi_{\ast}(s')$;
 
  \item\label{lem:properties_of_phi_ast(s'):injective_phi} Suppose that $\phi$ is
    injective. Then we get
    \[
      \phi_*(s')(g) =
      \begin{cases}
        \frac{\mu'(\pr'(K'))}{\mu(\pr(\phi(K')))} \cdot s(g') & \text{if} \; \phi(g') = g
        \;\text{for some }\; g' \in G'
        \\
        0 & g' \notin \im(\phi).
      \end{cases}
    \]
    and
    \[
      \supp_G(\phi_*(s')) = \phi(\supp_{G'}(s'));
    \]

  \item\label{lem:properties_of_phi_ast(s'):ring_homo}
    The map $\phi_* \colon \calh(G';R,\rho',\omega')_{\mu'}  \to \calh(G;R,\rho,\omega)_{\mu}$ is a homomorphism
    of (non-unital) rings.
  \end{enumerate}

\end{lemma}


\subsection{Approximate units}\label{subsec:Approximate_units}

    \begin{definition}[Rings with approximate units]\label{def:ring_with_approximate_unit}
      An \emph{approximate unit} for a ring $R$ is a subset $\{e_i \mid i \in I\}$ of
      elements $e_i \in R$ indexed by some directed set $I$ such that
      $e_i \cdot e_j = e_i = e_j \cdot e_i $ holds for $i \le j$ and for every element
      $r \in R$ there exists an index $i \in I$ with $e_i \cdot r = r = r \cdot e_i$.
    \end{definition}
    The ring $R$ has an approximate unit, if and only if there is a directed system of
    subrings $\{R_i \mid i \in I\}$ indexed by inclusion such that each $R_i$ is unital
    and $R = \bigcup_{i \in I} R_i$.  Obviously a unital ring has an approximate unit.
     
    Note that the ring $\calh(G;R,\rho,\omega)_{\mu}$ has a unit, if and only if $G$ is
    discrete. If $G$ is not discrete, $\calh(G;R,\rho,\omega)_{\mu}$ has at least an
    approximate unit by the following construction.
    
    Lemma~\ref{lem:condition_compatible_with_product} implies for  
    $K \in P$ that the subset
    \begin{equation}
      \calh(G//K;R,\rho,\omega)_{\mu} \subseteq \calh(G;R,\rho,\omega)_{\mu}
      \label{calh(G//K}
    \end{equation}
    consisting of those elements, for which $K$ is admissible, is closed under addition
    and multiplication and hence is a subring. Define an element $1_K$ in
    $\calh(G//K;R,\rho,\omega)_{\mu}$ by
    \begin{equation}
      \label{unit_in_calh(G//K}
      1_K(g) = \begin{cases}
        \frac{1}{\mu(\pr(K))} \cdot \omega(n)  &
        \text{if}\; g \in N\!K, g = nk \;\text{for}\; n \in N, k \in K;
        \\
        0  & \text{otherwise}.
      \end{cases}
    \end{equation}

\begin{lemma}\label{lem:calh(H)_as_union} 
  The element $1_{K}$ is a unit in $\calh(G//K;R,\rho,\omega)_{\mu}$. Moreover
  \begin{eqnarray*}
    \calh(G//K;R,\rho,\omega)_{\mu} 
    & \subseteq & 
                  \calh(G//K';R,\rho,\omega)_{\mu} \quad \text{if} \; K' \subseteq K;
    \\
    \calh(G;R,\rho,\omega)_{\mu}  
    & = & 
          \bigcup_{K} \calh(G//K;R,\rho,\omega)_{\mu},
  \end{eqnarray*}
  where $K$ and $K'$ run through the elements of $P$.
\end{lemma}


\subsection{Discarding $\mu$}\label{subsec:Discarding_mu}

In the sequel we omit the subscript $\mu$ in the notation of the Hecke algebra, since for
two $\IQ$-valued Haar measures $\mu $ and $\mu'$ on $G/N$ there is precisely one rational
number $r$ satisfying $r > 0$ and $\mu' = r \cdot \mu$, and the map
\[\calh(G;R,\rho,\omega)_{\mu'} \xrightarrow{\cong}
  \calh(G;R,\rho,\omega)_{\mu}, \quad s \mapsto r \cdot s.
\]
is an isomorphism of rings.


\typeout{----- Section 3: $\IZ$-categories, additive categories and idempotent completions  -------}

\section{$\IZ$-categories, additive categories and idempotent completions}%
\label{sec:Z-categories_additive_categories_and_idempotent_completions}

A \emph{$\IZ$-category} is a category $\cala$
such that for every two objects $A$ and $A'$ in $\cala$ the set of morphisms
$\mor_{\cala}(A,A')$ has the structure of a $\IZ$-module  and composition is
$\IZ$-bilinear.  If $G$ is a group, a \emph{$G$-$\IZ$-category} is a $\IZ$-category
with a left $G$-action by automorphisms of $\IZ$-categories. Note that we do not require
that $\cala$ has identity morphisms.  Given a
ring $R$, we denote by $\underline{R}$ the $\IZ$-category with precisely one object, whose
$\IZ$-module of endomorphisms is given by $R$ with its additive structure and
composition is given by the multiplication in $R$.  Obviously $\underline{R}$ is unital, if
and only if $R$ is unital.

An \emph{additive category} is a $\IZ$-category with finite direct sums.
Given a ring $R$, the category $\MODcat{R}_{\fgf}$ of finitely generated free $R$-modules 
carries an obvious structure of an additive category. Note that we do not require
that $\cala$ has identity morphisms. If it does,  we call it unital.

Given a $\IZ$-category $\cala$, let $\cala_{\oplus}$ be the associated additive category, whose
objects are finite tuples of objects in $\cala$ and whose morphisms are given by matrices
of morphisms in $\cala$ (of the right size) and the direct sum is given by concatenation
of tuples and the block sum of matrices, see for
instance~\cite[Section~1.3]{Lueck-Steimle(2016BHS)}. If $\cala$ is unital, $\cala_{\oplus}$ is unital.

Let $R$ be a unital ring. Then the obvious inclusion of unital additive categories 
\begin{equation}
\underline{R}_{\oplus} \xrightarrow{\simeq} \MODcat{R}_{\fgf}
\label{underline(r)_oplus_to_MODcat(R)_fgf}
\end{equation}
is an equivalence of unital additive categories.

Given an additive category $\cala$, its \emph{idempotent completion} $\Idem(\cala)$ is
defined to be the following additive category. Objects are morphisms $p \colon A \to A$ in
$\cala$ satisfying $p \circ p = p$.  A morphism $f$ from $p_1 \colon A_1 \to A_1$ to
$p_2 \colon A_2 \to A_2$ is a morphism $f \colon A_1 \to A_2$ in $\cala$ satisfying
$p_2 \circ f \circ p_1 = f$.  Note that $\Idem(\cala)$ is always unital, regardless
whether $\cala$ is unital or not.
The identity of an object $(A,p)$ is given by the morphism $p \colon (A,p) \to (A,p)$.

If $\cala$ is unital, then there is a obvious embedding
\[
\eta(\cala)\colon\cala \to \Idem(\cala)
\]
sending an object $A$ to
$\id_A \colon A \to A$ and a morphism $f \colon A \to B$ to the morphism given by $f$ again.  
A unital additive category $\cala$ is called
\emph{idempotent complete}, if $\eta(\cala)\colon \cala \to \Idem(\cala)$ is an
equivalence of unital additive categories, or, equivalently, 
if for every idempotent $p \colon A \to A$ in $\cala$
there are objects $B$ and $C$ and an isomorphism 
$f \colon A \xrightarrow{\cong} B \oplus C$ in $\cala$
such that $f \circ p \circ f^{-1} \colon B \oplus C \to B \oplus C$ is given by 
$\squarematrix{\id_B}{0}{0}{0}$.  The idempotent completion $\Idem(\cala)$ of a unital
additive category $\cala$ is idempotent complete.

Let $R$ be unital ring. Let $\MODcat{R}_{\fgp}$ be the unital additive category of
finitely generated projective $R$-modules. We obtain an equivalence of unital additive categories
$\Idem(\MODcat{R}_{\fgf}) \xrightarrow{\simeq} \MODcat{R}_{\fgp}$ by sending an object
$(F,p)$ to $\im(p)$. It and  the functor of~\eqref{underline(r)_oplus_to_MODcat(R)_fgf}
induce an equivalence of unital additive categories
\begin{equation}
\theta_R \colon \Idem\bigl(\underline{R}_{\oplus}\bigr) \xrightarrow{\simeq}  \MODcat{R}_{\fgp}.
\label{underline(r)_oplus_to_MODcat(R)_fgp}
\end{equation}

Let $\cala$ be an additive category. Let $\Phi \colon \cala \to \cala$ be an automorphism
of additive categories. Define the the additive category $\cala_{\Phi}[t,t^{-1}]$ called
\emph{$\Phi$-twisted finite Laurent category}
  as follows. It has the   same objects as $\cala$.  Given two objects $A$ and $B$, a morphism 
  $f  \colon A \to B$ in $\cala_{\Phi}[t,t^{-1}]$ is a formal sum 
  $f = \sum_{i \in    \IZ} f_i \cdot t^i$, where $f_i \colon \Phi^i(A) \to B$ is a morphism in
  $\cala$ from $\Phi^i(A)$ to $B$ and only finitely many of the morphisms $f_i$ are non-trivial. If 
  $g   = \sum_{j \in \IZ} g_j \cdot t^j$ is a morphism in $\cala_{\Phi}[t,t^{-1}]$
  from $B$ to $C$, we define the composite $g \circ f \colon A \to C$ by
  \[
  g \circ f := \sum_{k \in \IZ} \biggl( \sum_{\substack{i,j \in \IZ,\\i+j = k}}
    g_j \circ \Phi^j(f_i)\biggr) \cdot t^k.
  \]
If $\cala$ is unital, then $\cala_{\Phi}[t,t^{-1}]$ is unital again.

Let $R$ be a (not necessarily unital) ring with an automorphism
$\phi \colon R \xrightarrow{\cong} R$ of rings.  Let $R_{\phi}[t,t^{-1}]$ be the ring of
$\phi$-twisted finite Laurent series with coefficients in $R$.  We obtain from $\phi$ an
automorphism $\Phi \colon \underline{R} \xrightarrow{\cong} \underline{R}$ of
$\IZ$-categories.  There is an obvious isomorphism of $\IZ$-categories
\begin{equation}
\underline{R}_{\Phi}[t,t^{-1}] \xrightarrow{\cong} \underline{R_{\phi}[t,t^{-1}]}.
\label{underline_and_Laurent}
\end{equation}
If $R$ is unital, then we obtain equivalences  of unital additive categories
\begin{eqnarray}
(\underline{R}_{\oplus})_{\Phi}[t,t^{-1}] 
&\xrightarrow{\simeq}  &
\MODcat{R_{\phi}[t,t^{-1}]}_{\fgf};
\nonumber
  \\
\Idem\bigl((\underline{R}_{\oplus})_{\Phi}[t,t^{-1}]\bigr)
&\xrightarrow{\simeq}  &
\MODcat{R_{\phi}[t,t^{-1}]}_{\fgp}.
\label{Idem_and_Z_upper_d_and_rings}
\end{eqnarray}


\typeout{-------- Section 4: The algebraic $K$-theory of additive categories ---------------}

\section{The algebraic $K$-theory of $\IZ$-categories}%
\label{sec:The_algebraic_K-theory_of_Z-categories}

Let $\cala$ be a unital additive category. A construction of the \emph{non-connective $K$-theory
spectrum} $\bfKinfty(\cala)$ of a unital additive category can be found for instance
in~\cite{Lueck-Steimle(2014delooping)} or~\cite{Pedersen-Weibel(1985)}.
We get from the canonical embedding $\eta(\cala)\colon\cala \to \Idem(\cala)$
a weak homotopy equivalence
$\bfKinfty(\eta(\cala)) \colon \bfKinfty(\cala) \to \bfKinfty(\Idem(\cala))$ on the
non-connective $K$-theory, see for instance~\cite[Lemma~3.3~(ii)]{Bartels-Lueck(2020additive)}.

  \begin{definition}[Algebraic $K$-theory of (not necessarily unital) $\IZ$-categories]%
\label{def:Algebraic_K-theory_of_Z-categories}
    We will define the \emph{algebraic $K$-theory spectrum $\bfKinfty(\cala)$} of the (not necessarily
    unital) $\IZ$-category $\cala$ to be the non-connective algebraic $K$-theory spectrum of the
    unital additive  category $\Idem(\cala_{\oplus})$. Define for $n \in \IZ$
   \[
   K_n(\cala) := \pi_n(\bfKinfty(\cala)).
   \]
  \end{definition}

  Note that Definition~\ref{def:Algebraic_K-theory_of_Z-categories} extends the definition
  of the non-connective $K$-theory spectrum of unital additive categories to not
  necessarily unital $\IZ$-categories.

  A functor $F \colon \cala \to \cala'$ of (not necessarily unital) $\IZ$-categories
  induces a map of spectra
   \begin{equation}
   \bfKinfty(F) \colon \bfKinfty(\cala) \to \bfKinfty(\cala').
   \label{bfK(F)_colon_bfK(cala)_to_bfK(cala')}
   \end{equation}

   If the (not necessarily unital) $\IZ$-category $\cala$ is the directed
   union of (not necessarily unital) $\IZ$-subcategories $\cala_i$, then the canonical map
   \begin{equation}
   \hocolim_{i \in I} \bfKinfty(\cala_i) \xrightarrow{\simeq}  \bfKinfty(\cala)
   \label{hocolim_bfK(cala_i)_tobfK(cala)}
   \end{equation}
   is a weak homotopy equivalence and for every $n \in \IZ$ the canonical map
   \begin{equation}
   \colim_{i \in I} K_n(\cala_i) \xrightarrow{\cong} K_n(\cala)
   \label{colim_K_n(cala_i)_to_k_n(cala)}
   \end{equation}
   is a bijection.  We conclude~\eqref{hocolim_bfK(cala_i)_tobfK(cala)} and~\eqref{colim_K_n(cala_i)_to_k_n(cala)} 
   for instance from~\cite[Corollary~7.2]{Lueck-Steimle(2014delooping)}.

   If $R$ is an associative ring (not necessarily with a unit), we define the
   non-connective $K$-theory spectrum $\bfKinfty(R)$ to be $\bfKinfty(\underline{R})$ and
   $K_n(R) := \pi_n(\bfKinfty(R))$ for $n \in \IZ$.  If $R$ has an approximate unit, then
   our definition of $K_n(R)$ agrees with the usual definition of $K_n(R)$ for a ring
   without unit by the kernel of the map $K_n(R_+) \to K_n(\IZ)$, where $R_+$ is the ring
   with unit associated to $R$. Because of Lemma~\ref{lem:calh(H)_as_union} this applies
   to the Hecke algebra $\calh(G;R,\rho,\omega)$.


\typeout{------------------------- Section 5: Covirtually cyclic groups -------------------}

\section{Covirtually $\IZ$ groups}\label{subsec:covirtually_Z_groups}

Let $G$, $N$, $Q$, $R$, $\rho$, $P$, $\omega$, and $\mu$ as in
Subsection~\ref{subsec:normal_characters}.
In particular we can consider the Hecke algebra $\calh(G;R,\rho,\omega)$,
see  Subsection~\ref{subsec:The_construction_of_the_Hecke_algebra}.
Assume furthermore, that we have a normal open subgroup $L \subseteq G$ satisfying:

\begin{itemize}
\item $G/L$ is isomorphic to $\IZ$;
\item $N \subseteq L$;
\item $M := L/N$ is compact;
\end{itemize}

Note that we get exact sequences of td-groups $1 \to L \to G \to \IZ \to 1$
and $1 \to M \to Q \to \IZ \to 1$, where $\IZ$ is considered as discrete group
and $M$ is compact.

Let $g_0 \in G$ be any element which represents in $G/L$ a generator. Let
$\phi \colon L \to L$ be the automorphism of $L$ given by conjugation with $g_0$. Denote
by $L\rtimes_{c_{g_0}} \IZ$ the td-group given by the semi-direct
product of $L$ with the discrete group $\IZ$ with respect to $c_{g_0}$.  Then we get an
isomorphism of td-groups
\[
\alpha \colon L \rtimes_{c_{g_0}} \IZ \xrightarrow{\cong} G; \quad lt^n \mapsto lg_0^n,
\]
if $t \in \IZ$ is a fixed generator. It induces also an isomorphism
$\beta \colon M \rtimes_{c_{q_0}} \IZ \xrightarrow{\cong} Q$ , if we
put $q_0 = \pr(g_0)$. In the sequel we identify $G = L \rtimes_{c_{g_0N}} \IZ$ and 
$g_0$ with $e_Lt$ for $e_L \in L$ the unit and $Q = M \rtimes_{c_{g_0}} \IZ$ and 
$g_0N$ with $e_Qt$ for $e_Q \in Q$ the unit.

Since $L \subseteq G$ is open, the $\IQ$-valued measure $\mu$ on $G$ defines a $\IQ$-valued
measure on $L$ by restriction,  which we will denote by $\mu$ again. Note that we can
consider the Hecke algebra $\calh(L;R,\rho|_L,\omega)$.

Next we check that the automorphism $c_{g_0} \colon L \to L$ induces an automorphism of
rings
\begin{equation}
\phi \colon \calh(L;R,\rho|_L,\omega)\xrightarrow{\cong} \calh(L;R,\rho|_L,\omega)
\label{induced_auto_on_Hecke_algebra_of_L}
\end{equation}
by sending $s \in \calh(L;R,\rho|_L,\omega)$ given by a function $s \colon L \to R$
to the element given by the function $\phi(s) \colon L \to R, \;l \mapsto
ts(t^{-1}lt)$. Note that this is not just~\eqref{phi_ast_coming_from_open_phi} applied to
$c_{g_0}$, condition~\eqref{rho_is_rho'_circ_phi} $c_{g_0}$ is not satisfied for
$c_{g_0}$. So we have to check that $\phi(s)$ defines an element in
$\calh(L;R,\rho|_L,\omega)$.

Obviously the image of the  support of $\phi(s)$ under $L \to L/N$ is compact, since
this is true for $\supp(s)$ and $\supp(\phi(s)) = t\supp(s)t^{-1}$.

Suppose that $K \in P$ is admissible for $s$. Then $tKt^{-1}$ is
admissible for $\phi(s)$ by the following calculation for $l \in L$ and $k' \in tKt^{-1}$,
if we write $k' = tkt^{-1}$ for $k \in K$
\[
\phi(s)(k'l)
= 
ts(t^{-1}k'lt)
 = 
ts(t^{-1}tkt^{-1}lt)
= 
ts(kt^{-1}lt)
\stackrel{\eqref{condition_s_Hecke_algebra:left}}{=} 
t s(t^{-1}lt)
= \phi(s)(l),
\]
and 
\[
\phi(s)(lk') 
= 
ts(t^{-1}lk't)
=
ts(t^{-1}ltkt^{-1}t)
=  
ts(t^{-1}ltk)
\\
\stackrel{\eqref{condition_s_Hecke_algebra:right}}{=} 
ts(t^{-1}lt)
=
\phi(s)(l).
\]
The following calculation shows that
condition~\eqref{condition_s_and_omega:left} is satisfied.
\begin{multline*}
\phi(s)(nl)
 = 
  ts(t^{-1}nlt)
 =
 ts(t^{-1}ntt^{-1}lt)
 \stackrel{\eqref{condition_s_and_omega:left}}{=} 
 t\bigl( \omega(t^{-1}nt) \cdot s(t^{-1}lt)\bigr)
 \\
 =
 t \omega(t^{-1}nt) \cdot ts(t^{-1}lt)                                                             
\stackrel{\eqref{omega_and_conjugation},~\eqref{rho_and_omega}}{=} 
\omega(n) \cdot ts(t^{-1}lt)
= 
  \omega(n) \cdot \phi(s)(l).
\end{multline*}
Recall that the condition~\eqref{condition_s_and_omega:right} holds
automatically, see Remark~\ref{rem:redundance}. Hence $\phi$ is well-defined.

It is obviously compatible with the addition. It is compatible with the multiplication by
the following calculation for two elements $s,s' \in\calh(L;R,P|_L,\omega)$ and
$l \in L$, where $K \in P$ is admissible for both $s$ and $s'$, and $T$ is a transversal
for the projection $L \to L/N\!K$, and $\pr \colon L \to M = L/N$ is the projection.  We
will use the fact that $tTt^{-1}$ is a transversal for the projection $L \to L/NtKt^{-1}$
and $tKt^{-1}$ is admissible for $\phi(s)$ and $\phi(s')$. Moreover, we have
\begin{equation}
  [M:\pr(K)] = [tMt^{-1}:t\pr(K)t^{-1}] = [M :  \pr(tKt^{-1})].
  \label{indexM:pr_versus_M:pr(tkt_upper_(-1)}
\end{equation}
We compute
\begin{eqnarray*}
\phi(s\cdot s')(l)
& = & 
t(s \cdot s')(t^{-1}lt)
\\
& \stackrel{\eqref{calt_hecke_algebra:product}}{=} & 
t \left(\mu(\pr(K))\cdot \sum_{g'\in T} s(t^{-1}ltg') \cdot t^{-1}ltg's'(g'^{-1}) \right)
\\
& = & 
\mu(\pr(K)) \cdot \sum_{g'\in T} ts(t^{-1}ltg') \cdot ltg's'(g'^{-1}) 
\\
& = & 
\frac{\mu(M)}{[M:\pr(K)]}  \cdot \sum_{g'\in T} ts(t^{-1}ltg't^{-1}t) \cdot ltg't^{-1}ts'(t^{-1}tg'^{-1}t^{-1}t) 
\\
& \stackrel{\eqref{indexM:pr_versus_M:pr(tkt_upper_(-1)}}{=} & 
\frac{\mu(M)}{[M:\pr(tKt^{-1})]}\cdot \sum_{g''\in tTt^{-1}} ts(t^{-1}lg''t) \cdot lg''ts'(t^{-1}g''^{-1}t) 
\\
& = & 
\mu(\pr(tKt^{-1})) \cdot \sum_{g''\in tTt^{-1}} \phi(s)(lg'') \cdot lg''\phi(s')(g''^{-1})
\\
& \stackrel{\eqref{calt_hecke_algebra:product}}{=} & 
(\phi(s) \cdot \phi(s'))(l).
\end{eqnarray*}

\begin{lemma}\label{calh(L_rtimes_phi_Z_as_twisted_Laurent_ring}
There is a natural isomorphism of (non-unital) rings
\[
\Xi \colon \calh(L;R,\rho|_L,\omega)_{\phi}[t,t^{-1}] \xrightarrow{\cong} \calh(G;R,\rho,\omega).
\]
\end{lemma}
\begin{proof}
Consider an element $s \in \calh(L;R,\rho|_L,\omega)$ and an element $n \in \IZ$.
Then $\Xi(st^n)$ is  defined to be the element  in $\calh(G;R,\rho,\omega)$ given by 
\begin{equation}
G \to R, \quad (lt^m) \mapsto 
\begin{cases} s(l) & \text{if} \; m = n;
\\
0 & \text{otherwise.}
\end{cases}
\label{Xi(st_upper_n)}
\end{equation}
Obviously the image of the support of $\Xi(st^n)$ under $\pr \colon G \to Q$ is compact, as
it is a closed subset of $t^nM$ and $M \subseteq Q$ is compact.  Suppose that the compact
open subgroup $K \subseteq L$ is admissible for $s$. Then
$K \cap t^{-n}Kt^{n} \subseteq L \subseteq G$ is admissible for $\Xi(st^n)$
by  the following calculation for $l \in L$ and $k \in K \cap t^{-n}Kt^{n}$
\[
  \Xi(st^n)(klt^n) \stackrel{\eqref{Xi(st_upper_n)}}{=}    s(kl)
  \stackrel{\eqref{condition_s_Hecke_algebra:left}}{=}   s(l)
\stackrel{\eqref{Xi(st_upper_n)}}{=}   \Xi(st^n)(lt^n),
\]
and
\[
\Xi(st^n)(lt^nk) 
= 
\Xi(st^n)(lt^nkt^{-n}t^n) 
\stackrel{\eqref{Xi(st_upper_n)}}{=}   
s(lt^nkt^{-n}) 
\stackrel{\eqref{condition_s_Hecke_algebra:right}}{=} 
s(l) 
\stackrel{\eqref{Xi(st_upper_n)}}{=}   
\Xi(st^n)(lt^n)
\]
and the observation that we have
$\Xi(st^n)(lt^mk) = \Xi(st^n)(klt^m) = \Xi(st^n)(lt^m) =
0$ for $m \in \IZ$ with $m \not= n$. 
Next we verify condition~\eqref{condition_s_and_omega:left}. We get for
$z \in N$ and $m \in \IZ$ with $m \not= n$ that $\Xi(st^n)(zlt^m) = 0 = \Xi(st^n)(nlt^m)$
and
\[
  \Xi(st^n)(zlt^n) \stackrel{\eqref{Xi(st_upper_n)}}{=}    s(zl) \stackrel{\eqref{condition_s_and_omega:left}}{=}
  \omega(z) \cdot s(l) \stackrel{\eqref{Xi(st_upper_n)}}{=}   \omega(z) \cdot \Xi(st^n)(lt^n)
\]
hold. Recall that the condition~\eqref{condition_s_and_omega:right} holds
automatically, see Remark~\ref{rem:redundance}. Thus we have shown that $\Xi(st^n)$ is a
well-defined element in $\calh(G;R,\rho,\omega)$.

Define the image under $\Xi$ of an arbitrary element in
$\calh(L;R,\rho|_L,\omega)_{\phi}[t,t^{-1}] $ given by a finite sum
$\sum_{n \in \IZ} s_n t^n$ to be the element $\sum_{n \in \IZ} \Xi(s_n t^n)$ in
$\calh(G;R,\rho,\omega)$.  Obviously $\Xi$ is compatible with the addition. In order
to show that $\Xi$ is compatible with the multiplication, it suffices to show for
$s,s' \in \calh(L;R,\rho|_L,\omega)$, $l \in L$, and $m',n,n' \in \IZ$
\[
\bigl(\Xi(st^n) \cdot \Xi(s't^{n'})\bigr)(lt^m) = \Xi(st^n\cdot s't^{n'})(lt^m).
\]

Fix a compact open subgroup $K \subseteq G$ such that $K$ is admissible for both
$\Xi(st^n)$ and $\Xi(s't^{n'})$ and $t^nKt^{-n}$ is admissible for both $s$ and
$\phi^n(s)$. Consider a transversal $T'$ for the projection $L \to L/N\!K$. Then
$T =\{t^{m'}l' \mid  m' \in \IZ, l' \in T'\}$ is a transversal for the projection
$G \to G/N\!K$ and the map $\IZ \times T' \xrightarrow{\cong} T$ sending $(m',l')$ to
$t^{m'}l'$ is a bijection. Moreover, $t^nT't^{-n}$ is a transversal for the projection
$L \to L/Nt^{n}Kt^{-n}$.  We have
\begin{multline}
  \mu(\pr(t^nKt^{-n})) = \mu(t^n\pr(K)t^{-n})
  \\
  = \frac{\mu(M)}{[M:t^n\pr(K)t^{-n}]}
   \stackrel{\eqref{indexM:pr_versus_M:pr(tkt_upper_(-1)}}{=}  \frac{\mu(M)}{[M:\pr(K)]} = \mu(\pr(K)).
\label{fdgfiewurhtiv349}
\end{multline}
We compute
\begin{eqnarray*}
  \lefteqn{\bigl(\Xi(st^n) \cdot \Xi(s't^{n'})\bigr)(lt^m)}
  &&
  \\
  & \stackrel{\eqref{calt_hecke_algebra:product}}{=} &
  \mu(\pr(K)) \cdot \sum_{g' \in T} \Xi(st^n)(lt^mg') \cdot lt^mg'\Xi(s't^{n'})(g'^{-1})
  \\
  & = & 
        \mu(\pr(K)) \cdot \sum_{l' \in T'} \sum_{m' \in \IZ} \Xi(st^n)(lt^mt^{m'}l')
        \cdot lt^mt^{m'}l'\Xi(s't^{n'})((t^{m'}l')^{-1})
\\
& = &
      \mu(\pr(K)) \cdot \sum_{l' \in T'} \sum_{m' \in \IZ} \Xi(st^n)(lt^{m+m'}l't^{-m-m'}t^{m+m'})
      \cdot lt^{m+m'}l'\Xi(s't^{n'})(l'^{-1}t^{-m'})
\\
& \stackrel{\eqref{Xi(st_upper_n)}}{=} &
\mu(\pr(K)) \cdot \sum_{l' \in T'} \sum_{\substack{m' \in \IZ\\ m+m' = n, -m' = n'}} 
s(lt^{m+m'}l't^{-m-m'}) \cdot  lt^{m+m'}l's'(l'^{-1})
\\
& = &
\begin{cases}
\mu(\pr(K)) \cdot \sum_{l' \in T'} 
s(lt^{n}l't^{-n}) \cdot  lt^nl's'(l'^{-1}) & m = n+n'\\
0 & m \not= n+n'
\end{cases}
\\
& = &
\begin{cases}
\mu(\pr(K)) \cdot \sum_{l' \in T'} 
s(lt^{n}l't^{-n}) \cdot  lt^nl't^{-n}t^ns'(t^{-n}t^nl'^{-1}t^{-n}t^n) & m = n+n'\\
0 & m \not= n+n'
\end{cases}
\\
& = &
\begin{cases}
\mu(\pr(K)) \cdot \sum_{l' \in T'} 
s(lt^{n}l't^{-n}) \cdot  lt^nl't^{-n}\phi^n(s')(t^nl'^{-1}t^{-n}) & m = n+n'\\
0 & m \not= n+n'
\end{cases}
\\
& \stackrel{\eqref{fdgfiewurhtiv349}}{=} &
\begin{cases}
\mu(t^nKt^{-n}) \cdot \sum_{l'' \in t^nT't^{-n}} 
s(ll'') \cdot  ll''\phi^n(s')(l''^{-1}) & m = n+n'\\
0 & m \not= n+n'
\end{cases}
\\
& \stackrel{\eqref{calt_hecke_algebra:product}}{=}  &
\begin{cases}
(s \cdot \phi^n(s'))(l) & m = n+n'\\
0 & m \not= n+n'
\end{cases}
\\
& \stackrel{\eqref{Xi(st_upper_n)}}{=} &
\Xi(s \cdot \phi^n(s') \cdot t^{n+n'})(lt^m) 
\\
& = &
\Xi(st^n \cdot s't^{n'})(lt^m).
\end{eqnarray*}

Obviously $\Xi$ is injective. It remains to show that $\Xi$ is surjective. Any element in
$\calh(G;R,\rho, \omega)$ can be written as a sum of elements $s'$ for which the support
is contained in $Lt^n$ for some $n \in \IZ$. Hence it suffices to show that such $s'$ is in the
image.  Define $s \colon L \to R$ by $s(l) = s'(lt^n)$. Choose $K \in P$ such that both
$K$ and $t^{-n}Kt^n$ are admissible for $s'$. Obviously $K \subseteq L$ and
$t^{-n}Kt^n \subseteq L$. We have for $l \in L$ and $k \in K$ the equality
$s'(klt^n) \stackrel{\eqref{condition_s_Hecke_algebra:left}}{=} s'(lt^n)$, which
implies $s(kl) = s(l)$.  We also have
$ s'(lkt^n) = s'(lt^nt^{-n}kt^n)
\stackrel{\eqref{condition_s_Hecke_algebra:right}}{=} s'(lt^n) $ which implies
$s(lk) = s(l)$. Hence $K$ is admissible for $s$.
Condition~\eqref{condition_s_and_omega:left} follows from the calculation for $z \in N$.
\[
  s(zl) = s'(zlt^n) \stackrel{~\eqref{condition_s_and_omega:left}}{=}
  \omega(z) \cdot s'(lt^n)  = \omega(z) \cdot s(l).
\]

Recall that the condition~\eqref{condition_s_and_omega:right}  holds
automatically, see Remark~\ref{rem:redundance}. 
We conclude that $s$ defines an element in $\calh(L;R,\rho|_L,\omega)$ with $\Xi(st^n) = s'$. 
This finishes the proof of Lemma~\ref{calh(L_rtimes_phi_Z_as_twisted_Laurent_ring}.
\end{proof}

\begin{lemma}\label{lem:idempotent-completions_and_twisted_Laurent_series_non-unital}
Let $\cala$ be a (not necessarily unital) additive category, which is the directed union
$\cala = \bigcup_{i \in I} \cala_i$ of unital additive categories. Let
$\Phi \colon \cala \xrightarrow{\cong} \cala$ be an automorphism of (non-unital) additive
categories.

There is an equivalence of unital additive categories
\[
F \colon \Idem\bigl(\Idem(\cala)_{\Idem(\Phi)}[t,t^{-1}]\bigr) 
\xrightarrow{\simeq} \Idem\bigl(\cala_{\Phi}[t,t^{-1}]\bigr).
\]
\end{lemma}
\begin{proof}
  Recall that an object in $\Idem(\cala)$ is given by a pair $(A,p)$, where $A$ is an
  object in $\cala$ and $p \colon A \to A$ is a morphism in $\cala$ with $p \circ p = p$.
  Moreover, a morphism $f \colon (A,p) \to (A',p')$ in
  $\Idem(\cala)_{\Idem(\Phi)}[t,t^{-1}]$ is given by a finite sum
  $f = \sum_{j \in \IZ} f_j \cdot t^j$, where
  $f_j \colon \Idem(\Phi)^j(A,p) := (\Phi^j(A),\Phi^j(p)) \to (A',p') $ is a morphism in
  $\Idem(\cala)$.  Hence each $f_j$ is given by a morphism $f_j \colon \Phi^j(A) \to A'$
  satisfying $f_j = p' \circ f_j \circ \Phi^j(p)$. We conclude that the morphism
  $f \colon (A,p) \to (A',p')$ in $\Idem(\cala)_{\Idem(\Phi)}[t,t^{-1}]$ is the same as a
  morphism $f \colon A \to A'$ in $\cala_{\Phi}[t,t^{-1}]$ satisfying
  $(p'\cdot t^0) \circ f \circ (p\cdot t^0) = f$, since we get in $\cala_{\Phi}[t,t^{-1}]$
  \[
   (p'\cdot t^0) \circ f \circ (p\cdot t^0)  
   =  
   \sum_{j\in \IZ}  (p' \cdot t^0) \circ f_j  \cdot t^j \circ (p \cdot t^0)
    = \sum_{j\in \IZ}  \bigl(p' \circ f_j  \cdot \Phi^j(p)\bigr) \cdot t^j.
  \]
  Now an object in $\Idem\bigl(\Idem(\cala)_{\Idem(\Phi)}[t,t^{-1}]\bigr)$ is given by
 $\bigl((A,p),q\bigr)$, where $A$ is an object in $\cala$, $p \colon A \to A$ is a
 morphism in $\cala$ with $p \circ p = p$, and $q \colon (A,p) \to (A,p)$ is a morphism in
 $\Idem(\cala)_{\Idem(\Phi)}[t,t^{-1}]$ satisfying $q \circ q = q$.  The morphism $q$ is
 the same as a morphism $q \colon A \to A$ in $\cala_{\Phi}[t,t^{-1}]$ satisfying 
 $(p\cdot  t^0) \circ q \circ (p\cdot t^0) = q$ and $q \circ q = q$. Hence we can define 
 $F$ on  objects by
\[
F((A,p),q) = (A,q).
\]
Consider two objects $((A,p),q)$ and $((A',p'),q')$. A morphism
$f \colon ((A,p),q) \to ((A',p'),q')$ in
$\Idem\bigl(\Idem(\cala)_{\Idem(\Phi)}[t,t^{-1}]\bigr)$ is the same as a morphism
$f \colon (A,p) \to (A',p')$ in $\Idem(\cala)_{\Idem(\Phi)}[t,t^{-1}]$ satisfying
$q' \circ f \circ q = f$ and therefore the same as a morphism $f \colon A \to A'$ in
$\cala_{\Phi}[t,t^{-1}]$ satisfying $(p' \cdot t^0) \circ f \circ (p\cdot t^0) = f$ and
$q' \circ f \circ q = f$.

Hence we can define $F$ on morphisms by sending the morphism
$f \colon \bigl((A,p),q\bigr) \to \bigl((A',p'),q'\bigr)$ in
$\Idem\bigl(\Idem(\cala)_{\Idem(\Phi)}[t,t^{-1}]\bigr)$ to the morphism
$(A,q) \to (A',q')$ in $\Idem\bigl(\cala_{\Phi}[t,t^{-1}]\bigr)$ given by the morphism
$f \colon A \to A'$ in $\cala_{\Phi}[t,t^{-1}]$.  One easily checks that $F$ is compatible
with composition and sends identity morphisms to identity morphisms.

Next we show that the  map induced by $F$
\begin{multline*}
\mor_{\Idem(\Idem(\cala)_{\Idem(\Phi)}[t,t^{-1}])}\bigl(((A,p),q),((A',p'),q')\bigr)
\\
\to \mor_{\Idem(\cala_{\Phi}[t,t^{-1}])}\bigl((A,q),(A',q')\bigr)
\end{multline*}
is bijective. Obviously it is injective.  In order to show surjectivity, we have to show
for a morphism $f \colon (A,q) \to (A',q')$ in $\Idem(\cala)_{\Idem(\Phi)}[t,t^{-1}]$
satisfying $q' \circ f \circ q = f$ that $(p' \cdot t^0) \circ f \circ (p\cdot t^0) = f$
holds.  This follows from the following computation using
$(p\cdot t^0) \circ q \circ (p\cdot t^0) = q$, $q \circ q = q$, $p \circ p = p$,
$(p'\cdot t^0) \circ q' \circ (p'\cdot t^0) = q'$, $q' \circ q' = q'$, and
$p' \circ p' = p'$,
\begin{multline*}
(p' \cdot t^0) \circ f \circ (p\cdot t^0)
= 
(p' \cdot t^0) \circ q' \circ f \circ q \circ (p\cdot t^0)
\\
=  
(p' \cdot t^0) \circ (p'\cdot  t^0) \circ q' \circ (p'\cdot t^0)
\circ f \circ (p\cdot  t^0) \circ q \circ (p\cdot t^0) \circ (p\cdot t^0)
\\
=  
(p' \cdot t^0)  \circ q' \circ (p'\cdot t^0)   \circ f \circ (p\cdot  t^0) \circ q \circ (p\cdot t^0) 
=  q' \circ f \circ q
= 
f.
\end{multline*}
Finally we show that $F$ is surjective on objects. Consider any object $(A,q)$ in
$\Idem\bigl(\cala_{\Phi}[t,t^{-1}]\bigr)$. In order to show that $(A,q)$ is in the image
of $F$, we have to construct a morphism $p \colon A \to A$ in $\cala$ such that
$p \circ p = p$ holds in $\cala$ and $(p\cdot t^0) \circ q \circ (p\cdot t^0) = q$ holds in
$\cala_{\Phi}[t,t^{-1}]$.

We can write $q$ as a finite sum $q = \sum_{j\in \IZ } q_j \cdot t^j$ for morphism
$q_j \colon \Phi^j(A) \to A$ in $\cala$. Since $\cala$ is the directed union
$\bigcup_{i \in I} \cala_i$ of the unital subcategories $\cala_i$, we can find an index
$i_0 \in I$ such that for each $j \in \IZ$ with $q_j \not= 0$ and hence for all $j \in J$
the morphisms $q_j$ and $\Phi^{-j}(q_j)$ belong to $\cala_{i_0}$.  Let $p \in \cala_{i_0}$
be the identity morphism of the object $A$ in $\cala_{i_0}$. Then we get $p \circ p = p$,
$p \circ q_j = q_j$, and $\Phi^{-j}(q_j) \circ p = \Phi^{-j}(q_j)$ in $\cala$ for all
$j \in \IZ$. Now we compute
\begin{multline*}
(p\cdot t^0) \circ q \circ (p\cdot t^0) 
= 
(p\cdot t^0) \circ \left(\sum_{j\in \IZ } q_j \cdot t^j \right) \circ (p\cdot t^0) 
\\
= 
\sum_{j\in \IZ } (p\cdot t^0) \circ (q_j \cdot t^j) \circ (pt^0) 
=  
\sum_{j\in \IZ } (p\circ q_j \cdot \Phi^j(p)) \cdot t^j
=  
\sum_{j\in \IZ } (q_j \cdot \Phi^j(p)) \cdot t^j
\\
=  
\sum_{j\in \IZ } \Phi^j(\Phi^{-j}(q_j )\cdot p) \cdot t^j
= 
\sum_{j\in \IZ } \Phi^j(\Phi^{-j}(q_j)) \cdot t^j
=  
\sum_{j\in \IZ } q_j\cdot t^j 
= q.
\end{multline*}
This finishes the proof of
Lemma~\ref{lem:idempotent-completions_and_twisted_Laurent_series_non-unital}
\end{proof}

The next lemma allows to reduced the computation of the algebraic $K$-theory of the
non-unital ring $\calh(G;R,\rho,\omega)$ to the calculation of the algebraic
$K$-theory of a unital additive category given by the twisted finite Laurent category of
an automorphism of a unital additive category.  The main advantage will be that for such a
category Bass-Heller-Swan decompositions will be available.

\begin{lemma}\label{lem:K-theory_of_virt_Z_reduced_to_Laurent}
There is a weak equivalence 
\[
\bfKinfty\bigl(\Idem(\underline{ \calh(L;R,\rho|_L,\omega)}_{\oplus})_{\Idem(\underline{\phi}_{\oplus})}[t,t^{-1}]\bigr)
\xrightarrow{\simeq}
\bfKinfty\bigl(\calh(G;R,\rho,\omega)\bigr). 
\]
\end{lemma}
\begin{proof} Recall that for a unital additive category $\calb$ the obvious map
  $\bfKinfty(\calb) \to \bfKinfty(\Idem(\calb))$ is a weak homotopy equivalence. We can apply
Lemma~\ref{lem:idempotent-completions_and_twisted_Laurent_series_non-unital}  to
$\cala = \underline{ \calh(L;R,\rho|_L,\omega)}_{\oplus}$ and the automorphism
$\underline{\phi}_{\oplus}$ because of Lemma~\ref{lem:calh(H)_as_union}. Hence we obtain a
weak equivalence
\begin{multline*}
\bfKinfty\bigl(\Idem(\underline{\calh(L;R,\rho|_L,\omega)}_{\oplus})_{\Idem(\underline{\phi}_{\oplus})}[t,t^{-1}]\bigr)
\\
\xrightarrow{\simeq}
\bfKinfty\bigl(\Idem\bigl((\underline{\calh(L;R,\rho|_L,\omega)}_{\oplus})_{\underline{\phi}_{\oplus}}[t,t^{-1}]\bigr)\bigr).
\end{multline*}
The (non-unital) additive category
$(\underline{
  \calh(L;R,\rho|_L,\omega)}_{\oplus})_{\underline{\phi}_{\oplus}}[t,t^{-1}]$ is
isomorphic to the (non-unital) additive category
$\bigl(\underline{ \calh(L;R,\rho|_L,\omega)_{\phi}[t,t^{-1}]}\bigr)_{\oplus}$
by~\eqref{underline_and_Laurent}, and hence by
Lemma~\ref{calh(L_rtimes_phi_Z_as_twisted_Laurent_ring} to the (non-unital) additive
category $\underline{\calh(G;R,\rho, \omega)}_{\oplus}$.  Hence we obtain a weak homotopy
equivalence
\[\bfKinfty\bigl(\Idem\bigl((\underline{ \calh(L;R,\rho|_L,\omega)}_{\oplus})_{\underline{\phi}_{\oplus}}[t,t^{-1}]\bigr)\bigr)
\xrightarrow{\simeq} 
\bfKinfty\bigl(\Idem\bigl(\underline{\calh(G;R,\rho,\omega)}_{\oplus}\bigr)\bigr).
\]
\end{proof}


\typeout{----- Section 6: A review of the twisted  Bass-Heller-Swan decomposition for additive categories  -------}

\section{A review of the twisted Bass-Heller-Swan decomposition for unital additive categories}%
\label{sec:A_review_of_the_Bass-Heller-Swan_decomposition_for_unital_additive_categories}

In this section additive category means always a small unital additive
category and functors are assumed to respect identity morphisms. The same is true for rings.

The following definitions are  taken from~\cite[Definition~6.1]{Bartels-Lueck(2020additive)}.

\begin{definition}[Regularity properties of rings]\label{def:Regularity_properties_of_rings}

Let $l$ be a natural number.

\begin{enumerate}

\item\label{def:Regularity_properties_of_rings;Noetherian}
We call $R$ \emph{Noetherian}, if  any $R$-submodule of a finitely generated
$R$-module is again finitely generated;

\item\label{def:Regularity_properties_of_rings:regular_coherent}
We call $R$  \emph{regular coherent}, if  every finitely presented
$R$-module $M$ is of type $\FP$;

\item\label{def:Regularity_properties_of_rings:l-uniformly_regular_coherent}
We call $R$  \emph{$l$-uniformly regular coherent}, if  every finitely presented
$R$-module $M$ admits an $l$-dimensional finite projective
  resolution, i.e., there exist an exact sequence
  $0 \to P_l \to P_{l-1} \to \cdots \to P_0 \to M \to 0$ such that each $P_i$ is finitely
  generated projective;

\item\label{def:Regularity_properties_of_rings:regular}
We call $R$ \emph{regular}, if  it is Noetherian and regular coherent;

\item\label{def:Regularity_properties_of_ringss:l-uniformly_regular}
  We call $R$  \emph{$l$-uniformly regular}, if  it is Noetherian and $l$-uniformly regular coherent.

\end{enumerate}
\end{definition}

These notions are generalized to additive categories
in~\cite[Section~6]{Bartels-Lueck(2020additive)} in such a way that they reduce in the
special case $\cala = \underline{R}$ to the ones appearing in
Definition~\ref{def:Regularity_properties_of_rings}. Therefore the precise definitions for
additive categories are not needed to comprehend the material of this paper.

The following result follows from~\cite[Theorem~7.8 and Theorem~10.1]{Bartels-Lueck(2020additive)}.

\begin{theorem}[The non-connective $K$-theory of additive categories]%
\label{the:The_non_connective_K-theory_of_additive_categories}
Let $\cala$ be an  additive category.  Suppose that $\cala$
is regular.  Consider any automorphism
$\Phi \colon \cala \xrightarrow{\cong} \cala$ of additive categories.

Then we get  a weak homotopy equivalence of non-connective spectra
\[
\bfa^{\infty}  \colon \bfT_{\bfKinfty(\Phi^{-1})}  \xrightarrow{\simeq}  \bfKinfty(\cala_{\Phi}[t,t^{-1}]),
\]
where $\bfT_{\bfKinfty(\Phi^{-1})} $ is the mapping torus of the map of spectra 
$\bfKinfty(\Phi) \colon \bfKinfty(\cala ) \to \bfKinfty(\cala)$
induced by $\Phi$.
\end{theorem}


\typeout{--- Section 7:  Hecke algebras  over compact td-groups and crossed product rings ----}

\section{Hecke algebras over compact td-groups and crossed product rings}%
\label{sec:Hecke_algebras_over_compact_td-groups_and_crossed_product_rings}

Let $G$, $N$, $Q := G/N$, $\pr \colon G \to Q$, $R$, $P$, $\rho$, $\omega$, and
$\mu$ be as in Subsection~\ref{subsec:normal_characters}
and denote by $\calh(G;R,\rho,\omega)$ the Hecke algebra, which we have
introduced in Subsection~\ref{subsec:The_construction_of_the_Hecke_algebra}. Our
main assumption in this section will be that $Q$ is compact.

\begin{definition}\label{def_locally_central} We call a subgroup $N \subseteq G$
  \emph{locally central}, if the centralizer $C_GN$ of $N$ in $G$ is an open subgroup.
\end{definition}

The main result of this section is

\begin{theorem}\label{the:regular_coherence_of_the_Hecke_algebra}
  
  Suppose that $Q$ is compact and $N$ is locally central. Let $l$ be a natural number.
  Let $R$ be a unital ring with $\IQ \subseteq R$ such that $R$ is
  $l$-uniformly regular or regular respectively.
  
  Then the additive category
  $\Idem\bigl(\underline{\calh(G;R,\rho,\omega)[\IZ^m]}_{\oplus}\bigr)$
  is $(l+2m)$-uniformly regular or regular respectively for all $m \ge 0$. 
\end{theorem}

For the purpose of this paper we need Theorem~\ref{the:regular_coherence_of_the_Hecke_algebra}
only for the property regular, but for later applications it will be crucial to consider
the property $l$-uniformly regular as well. The point will be that the property
$l$-uniformly regular is compatible with infinite products of additive categories,
in contrast to the property regular.


\subsection{Existence of normal $K \in P$}\label{subsec:existence_of_normal_K_in_P}

\begin{lemma}\label{lem:arranging_K_to_be_normal}
  Suppose that $Q$ is compact and $N$ is locally central.

  Then for every compact open subgroup  $K \subseteq G$ there exists a compact open subgroup
  $K' \subseteq G$ such that $K' \subseteq K$,  $K' \subseteq C_GN$,  and $K'$ is normal in $G$.
 \end{lemma}
 \begin{proof}
   Put $L = K \cap C_GN$. Then $L$ is a compact open subgroup of $G$ satisfying
   $L \subseteq K$ and $L \subseteq C_GN$. Choose a transversal $T$ of the projection
   $G \to G/N\!L = Q/\pr(L)$.  Define $K' = \bigcap_{t \in t}tLt^{-1}$. Since $Q/\pr(L)$
   is compact and discrete, the set $T$ is finite. Hence $K'\subseteq G$ is again compact
   open. We get for $n \in N$ and $l \in L$
   \[
   (tnl)L(tnl)^{-1} = tnlLl^{-1}n^{-1}t^{-1} = tnLn^{-1}t^{-1} \stackrel{L \subseteq C_GN}{=} tLt^{-1}.
  \]
  This implies $K' = \bigcap_{g \in G}gLg^{-1}$. Hence $K' \subseteq G$ is a compact open
  normal subgroup and obviously satisfies $K' \subseteq K$ and $K' \subseteq C_GN$.  
\end{proof}


\subsection{Crossed products of finite groups and regularity}
\label{subsec:Crossed_products_of_finite_groups_and_regularity}

Let $R$ be a unital ring and $D$ be a (discrete) group.  Recall that a \emph{crossed
  product ring $R \ast D$} is a unital ring, which is a free left $R$-module with an
$R$-basis $\{b_d \mid d \in D\}$ indexed by the elements in $D$ such that $b_e$ is the
unit in $R \ast D$, for $d_1,d_2 \in D$ there is a unit $w(d_1,d_2) \in R^{\times}$
satisfying $b_{d_1d_2} = w(d_1,d_2) \cdot b_{d_1} \cdot b_{d_2}$, and for $r \in R$ and
$d \in D$ there exists $c_d(r) \in R$ with $c_d(r) \cdot b_d = b_d \cdot (r \cdot b_e)$,
where $c_e(r) = r$ is required for $r \in R$.  In particular each element $b_d$ has an
inverse $b_d^{-1}$ in $R \ast D$, (which is \emph{not} given by $b_{d^{-1}}$ in general,)
and there is an inclusion of rings $R \to R \ast D$ sending $r$ to $r \cdot b_e$.

The notion of crossed product ring is a generalization of the notion of a twisted group ring,
which is the special case, where $w$ is trivial.  For
more details we refer for instance to~\cite[Section~4]{Bartels-Lueck(2009coeff)}
or~\cite[Section~6]{Bartels-Reich(2007coeff)}.

\begin{lemma}\label{lem:regular_passes_to_finite_crossed_product_rings}
  Let $R$ be a ring with $\IQ \subseteq R$ and $D$ be a finite group. Let $R \ast D$ be a
  crossed product ring.
\begin{enumerate}
\item\label{lem:regular_passes_to_finite_crossed_product_rings:modules} Let $M$ be any
  $R \ast D$-module. Let $j \colon R \to R \ast D$ be the canonical inclusion of rings.
  Then we obtain $R \ast D$-homomorphisms
\begin{align*}
  i \colon M \to R \ast D \otimes_R j^*M, &
                   \quad x \mapsto \sum_{d \in D} \frac{1}{|D|} \cdot b_d \otimes b_d^{-1}\cdot x;
\\
  p \colon R \ast D \otimes_R j^*M \to M,
                                          & \quad u \otimes y \mapsto u \cdot y,
\end{align*}
satisfying $p \circ i = \id_M$, where $b_d^{-1}$ denotes  the inverse of $b_d$ in $R \ast D$;

\item\label{lem:regular_passes_to_finite_crossed_product_rings:regular} 
  If $R$ is regular, then $R \ast D$ is regular;

\item\label{lem:regular_passes_to_finite_crossed_product_rings:uniformly_regular} 
  If $R$ is $l$-uniformly regular, then $R \ast D$ is $l$-uniformly regular;

\item\label{lem:regular_passes_to_finite_crossed_product_rings:semi-simple} 
If $R$ is semi-simple, then $R \ast D$ is semi-simple.
\end{enumerate}
\end{lemma}
\begin{proof}~\ref{lem:regular_passes_to_finite_crossed_product_rings:modules} We check
  that $i$ is $R \ast D$-linear. Obviously $i$ is compatible with addition, it remains to
  treat multiplication.  Consider $r \in R$ and $d_0 \in D$.  Note for the sequel that
  the element $b_d^{-1}\cdot r\cdot b_{d_0}\cdot b_{d_0^{-1}d}$ in $R \ast D$ belongs to
  $R$. Hence we get for $x \in M$, $r \in R$ and $d_0 \in D$
\begin{multline*}
i(r\cdot b_{d_0} \cdot x) 
=  
\sum_{d \in D} \frac{1}{|D|} \cdot b_d \otimes b_d^{-1}\cdot (r\cdot b_{d_0} \cdot x)
=  
      \sum_{d \in D} \frac{1}{|D|} \cdot b_d \otimes (b_d^{-1}\cdot r \cdot b_{d_0}
      \cdot b_{d_0^{-1}d})\cdot (b_{d_0^{-1}d})^{-1}\cdot x
\\
=  
\sum_{d \in D} \frac{1}{|D|} \cdot b_d\cdot (b_d^{-1}\cdot r \cdot b_{d_0}
      \cdot  b_{d_0^{-1}d}) \otimes (b_{d_0^{-1}d})^{-1}\cdot x
= 
      \sum_{d \in D} \frac{1}{|D|} \cdot r\cdot b_{d_0}\cdot  b_{d_0^{-1}d} \otimes (b_{d_0^{-1}d})^{-1}
      \cdot x
\\
=  
      r \cdot b_{d_0} \cdot \frac{1}{|D|} \cdot  \sum_{d \in D} b_{d_0^{-1}d} \otimes (b_{d_0^{-1}d})^{-1}
      \cdot x
=  
r \cdot b_{d_0} \cdot \frac{1}{|D|} \cdot \sum_{d' \in D}  b_{d'} \otimes (b_{d'})^{-1}\cdot x
=
r \cdot b_{d_0} \cdot i(x).
\end{multline*}
Obviously $p$ is a well-defined $R \ast D$-homomorphism satisfying $p \circ i = \id_M$.
\\[1mm]~\ref{lem:regular_passes_to_finite_crossed_product_rings:regular} Since $R$ is regular,
$R$ is in particular Noetherian. Since $R \ast D$ is a finitely generated $R$-module,
$R \ast D$ is Noetherian as well.

It remains to show that a finitely presented $R \ast D$-module $M$ is
of type $\FP$. Since $R$ is regular and the $R$-module $i^*M$ is
finitely presented, $i^*M$ is of type $\FP$.  Since $R \ast D$ is free
as $R$-module and hence the functor sending an $R$-module $N$ to the
$R \ast D$-module $R \ast D \otimes_R N$ is flat and sends finitely
generated projective $R$-modules to finitely generated projective
$R \ast D$-modules, the $R \ast D$-module $R\ast D \otimes_R i^*M$ is
of type $\FP$. Since a direct summand in a module of type $\FP$ is of
type $\FP$ again, the $R \ast D$-module $M$ is of type $\FP$.
\\[1mm]~\ref{lem:regular_passes_to_finite_crossed_product_rings:uniformly_regular}
The proof is analogous to
assertion~\ref{lem:regular_passes_to_finite_crossed_product_rings:regular},
since all the statements about finite-dimension remain true, if one
inserts $l$-dimensional everywhere.
\\[1mm]~\ref{lem:regular_passes_to_finite_crossed_product_rings:semi-simple}
This follows from
assertion~\ref{lem:regular_passes_to_finite_crossed_product_rings:modules}.
This finishes the proof of
Lemma~\ref{lem:regular_passes_to_finite_crossed_product_rings}.
\end{proof}

\subsection{The  Hecke algebra and crossed products}%
\label{subsec:The_Hecke_algebra_and_crossed_products}

In this subsection we will assume that $Q$ is compact. 

Consider a compact open normal subgroup $K$ of $G$ satisfying  $K \in P$.
Since both $K$ and $N$ are normal in $G$, the subgroup $N\!K$ of  $G$ is also normal. Put
\begin{equation}
 D : = G/N\!K = Q/\pr(K).
\label{finite_discrete_group_D}
\end{equation}
Note that $D$ is a finite discrete group.

Next we show that $\calh(G;R,\rho,\omega)$ is a left $R$-module. Namely, define for 
$s \in \calh(G;R,\rho,\omega)$ the new element $rs$  by $rs(g) := r \cdot s(g)$
One easily checks  that $rs$ satisfies~\eqref{condition_s_and_omega:left},%
~\eqref{condition_s_and_omega:right},~\eqref{condition_s_Hecke_algebra:left},
and~\eqref{condition_s_Hecke_algebra:right}.

Fix a set-theoretic section  $\sigma \colon D \to G$ of the projection
$p \colon G \to D= G/N\!K$ satisfying $\sigma(e_D) = e_G$. In the sequel we denote by $T$
the transversal of $p$ given by $T := \{ \sigma(d)^{-1} \mid d \in D\}$.  For $d \in D$
define  $b_d \in \calh(G//K;R,\rho,\omega)$ by the function
\begin{multline}
b_d \colon G \to R, \quad 
\\
g \mapsto
\begin{cases}
\frac{1}{\mu(\pr(K))} \cdot \omega(n) 
& \text{if}\; p(g) = d \;\text{and}\; g = nk\sigma(d)\;\text{for}\; n \in N, k \in K;
\\
0 & p(g) \not= d.
\end{cases}
\label{b_d}
\end{multline}
This is independent of the choice of $n \in N$ and $k \in K$, since for
$n_0,n_1 \in N$ and $k_0,k_1 \in K$ with $n_0k_0 = n_1k_1$ we have 
$n_1^{-1}n_0 = k_1k_0^{-1} \in N \cap K$ and we compute
  \begin{equation}
    \label{choice_of_n-k_does_not_matter}
    \omega(n_1)
    =
     \omega(n_1) \cdot \omega(n_1^{-1}n_0) 
    \stackrel{\eqref{tau_and_omega}}{=} \omega(n_1) \cdot 1 =
    \omega(n_0).
  \end{equation}

  We have to check that the required transformation
  formulas~\eqref{condition_s_Hecke_algebra:left}
  and~\eqref{condition_s_Hecke_algebra:right} for $g \in G$ and $k \in K$ are
  satisfied.  If $p(g) \not= d$, then $b_d(kg) = b_d(g) = b_d(gk) = 0$ and the formulas
  hold. It remains to treat the case $p(g) = d$. This follows from the calculations for
  $g = nk\sigma(d)$ for $n \in N$, $k \in K$ and $k' \in K$ using
  $\sigma(d)k'\sigma(d)^{-1} \in K$
\begin{multline*}
b_d(k'g)
= 
b_d(k'nk\sigma(d))
= 
b_d((k'nk'^{-1})(k'k)\sigma(d))
\\
\stackrel{\eqref{b_d}}{=}  
\frac{1}{\mu(\pr(K))} \cdot \omega(k'nk'^{-1})
\stackrel{\eqref{omega_and_conjugation}}{=} 
\frac{1}{\mu(\pr(K))} \cdot \omega(n)
\stackrel{\eqref{b_d}}{=} b_d(g),
\end{multline*}
and
\begin{multline*}
b_d(gk')
= 
b_d(nk\sigma(d)k')
=  
b_d(n(k\sigma(d)k'\sigma(d)^{-1})\sigma(d))
\\
\stackrel{\eqref{b_d}}{=}  
\frac{1}{\mu(\pr(K))} \cdot \omega(n) 
\stackrel{\eqref{b_d}}{=}
b_d(g).
\end{multline*}
The verification of~\eqref{condition_s_and_omega:left}
and~\eqref{condition_s_and_omega:right} is left to the reader.  This finishes the proof
that $b_d$ is a well-defined element in $\calh(G//K;R,\rho,\omega)$.

Consider any element $s \in \calh(G//K;R,\rho,\omega)$. Then we get
\begin{eqnarray}
s & = & \sum_{d \in D} \mu(\pr(K)) \cdot s(\sigma(d)) \cdot b_d
\label{s_is_sum_d_s(sigma(d)_cdot_b_d}
\end{eqnarray}
by the following calculation for $g \in G$  with $g = nk\sigma(d)$ for $n \in N$ and $k \in K$
\begin{multline*}
s(g) 
= 
s(nk\sigma(d))
\stackrel{~\eqref{condition_s_and_omega:left},\eqref{condition_s_Hecke_algebra:left}}{=} 
  \omega(n) \cdot s(\sigma(d))
\\
 \stackrel{\omega(n) \in \cent(R)}{=} 
\mu(\pr(K))\cdot s(\sigma(d)) \cdot \left(\frac{1}{\mu(\pr(K))} \cdot \omega(n)\right)
\stackrel{\eqref{b_d}}{=} 
 \mu(\pr(K)) \cdot s(\sigma(d)) \cdot b_d(g).
\end{multline*}
We conclude from~\eqref{s_is_sum_d_s(sigma(d)_cdot_b_d} that
$\{b_d \mid d \in D\}$ is an $R$-basis for the left
$R$-module $\calh(G//K;R,\rho,\omega)$.

For $d_1, d_2$ in $D$, define an element 
\begin{eqnarray}
w(d_1,d_2) & := &
\omega(n) \in R^{\times},
\label{calt(d_1,d_2)}
\end{eqnarray}
if $\sigma(d_1d_2)\sigma(d_2)^{-1}\sigma(d_1)^{-1} = nk$ for $n \in N$ and $k\in K$.
This is independent of the choice of $n \in N$ and $k\in K$ by~\eqref{choice_of_n-k_does_not_matter}.
Next we want to show
\begin{eqnarray}
b_{d_1} \cdot b_{d_2} 
& = &
w(d_1,d_2) \cdot b_{d_1d_2}.
\label{b_d_1_cdot_b_d_2_is_omega_cdot_b_d_d_2}
\end{eqnarray}
Consider $d_1,d_2 \in D$ and $g \in G$. Choose elements  $n \in N$ and $k\in K$ satisfying
$\sigma(d_1d_2)\sigma(d_2)^{-1}\sigma(d_1)^{-1} = nk$.
If $p(g) = d_1d_2$, we fix $n_0 \in N$  and $k_0 \in K$
satisfying $g = n_0k_0 \sigma(d_1) \sigma(d_2)$.
We compute 
\begin{eqnarray}
  & &
 \label{aux_comp_b_d:(1)}
 \\
\lefteqn{(b_{d_1} \cdot b_{d_2})(g)}
  & &
  \nonumber
  \\
& \stackrel{\eqref{calt_hecke_algebra:product}}{=} & 
\mu(\pr(K)) \cdot \sum_{d \in D} b_{d_1}(g\sigma(d)^{-1}) \cdot g\sigma(d)^{-1}b_{d_2}(\sigma(d))
  \nonumber
  \\
  & \stackrel{\eqref{b_d}}{=} & 
\mu(\pr(K))\cdot \sum_{d \in D, p(\sigma(d)) = d_2} b_{d_1}(g\sigma(d)^{-1}) \cdot g\sigma(d)^{-1}b_{d_2}(\sigma(d))
   \nonumber
  \\
  &  = & 
\mu(\pr(K)) \cdot  b_{d_1}(g\sigma(d_2)^{-1}) \cdot g\sigma(d_2)^{-1}b_{d_2}(\sigma(d_2))
   \nonumber
  \\
  & \stackrel{\eqref{b_d}}{=} & 
\mu(\pr(K)) \cdot  b_{d_1}(g\sigma(d_2)^{-1}) \cdot g\sigma(d_2)^{-1} \cdot \left(\frac{1}{\mu(\pr(K))} \cdot \omega(e) \right)\nonumber
  \\
  & = &
  b_{d_1}(g\sigma(d_2)^{-1}) \cdot \bigl(g\sigma(d_2)^{-1} \cdot 1\bigr)
  \nonumber  
   \\
  & = &
  b_{d_1}(g\sigma(d_2)^{-1}) 
  \nonumber      
  \\
  & \stackrel{\eqref{b_d}}{=} & 
\begin{cases}
\frac{1}{\mu(\pr(K))} \cdot \omega(n_0) 
& \text{if}\; p(g) = d_1d_2;
\\
= 0 & \text{if}\; p(g) \not= d_1d_2.
\end{cases}
\nonumber
\end{eqnarray}
Suppose for $g \in G$ that $p(g) = d_1d_2$. We can write 
\[
  g = n_0k_0 \sigma(d_1) \sigma(d_2) = \bigl(n_0k_0k^{-1}n^{-1}(k_0k^{-1})^{-1}\bigr)\bigl(k_0k^{-1}\bigr)\sigma(d_1d_2)
\]
and have $n_0k_0k^{-1}n^{-1}(k_0k^{-1})^{-1} \in N$ and $k_0k^{-1} \in K$. We compute
\begin{eqnarray}
  & &
  \label{aux_comp_b_d:(2)}
  \\
  \lefteqn{w(d_1,d_2) \cdot b_{d_1d_2}(g)}
  & &
  \nonumber
  \\
  & \stackrel{\eqref{calt(d_1,d_2)}}{=} &
 \omega(n) \cdot b_{d_1d_2}(g)
 \nonumber
  \\
  & \stackrel{\eqref{b_d}}{=}   &
        \omega(n)  \cdot  \frac{1}{\mu(\pr(K))}
        \cdot \omega(n_0k_0k^{-1}n^{-1}(k_0k^{-1})^{-1}) 
\nonumber
 \\
  & = &
        \omega(n) \cdot  \frac{1}{\mu(\pr(K))} \cdot \omega(n_0)
        \cdot \omega(k_0k^{-1}n^{-1}(k_0k^{-1})^{-1}) 
 \nonumber
 \\
  & \stackrel{\eqref{omega_and_conjugation}}{=} &
                                                  \omega(n) \cdot  \frac{1}{\mu(\pr(K))} \cdot \omega(n_0)
                                                  \cdot \omega(n^{-1}) 
\nonumber
 \\
 & \stackrel{\omega(n_0) \in \cent(R)}{=}& 
                                                 \frac{1}{\mu(\pr(K))} \cdot  \omega(n) \cdot  \omega(n^{-1}) \cdot \omega(n_0)
\\
& = &
\frac{1}{\mu(\pr(K))} \cdot  \omega(n_0).
\nonumber                                                                 
\end{eqnarray}
Since $w(d_1,d_2) \cdot b_{d_1d_2}(g) = 0$, if $p(g) \not=d_1d_2$, we
conclude~\eqref{b_d_1_cdot_b_d_2_is_omega_cdot_b_d_d_2}
from~\eqref{aux_comp_b_d:(1)} and~\eqref{aux_comp_b_d:(2)}.

We compute for  $d\in D$, $r \in R$ and $d' \in D$ using the fact
that $\{\sigma(d'')^{-1} \mid d'' \in D\}$ is a transversal for $G \to G/N\!K = D$ and $\sigma(e_D) = e_G$

\begin{eqnarray*}
\lefteqn{\bigl(b_d \cdot (r\cdot b_{e_D})\bigr)(\sigma(d'))}
& & 
\\
&  \stackrel{\eqref{calt_hecke_algebra:product}}{=} &
\mu(\pr(K)) \cdot \sum_{d''\in D} b_d(\sigma(d')\sigma(d'')^{-1})
\cdot \sigma(d') \sigma(d'')^{-1}\bigl((r \cdot b_{e_D})(\sigma(d''))\bigr)
  \\
  & \stackrel{\eqref{b_d}}{=} &
\mu(\pr(K)) \cdot \sum_{d''\in \{e_Q\}} b_d(\sigma(d')\sigma(d'')^{-1})
\cdot \sigma(d') \sigma(d'')^{-1}\bigl((r \cdot b_{e_D})(\sigma(d''))\bigr)
\\
& = &
  \mu(\pr(K)) \cdot b_d(\sigma(d')e_G^{-1}) \cdot \sigma(d')e_G^{-1}\bigl((r \cdot b_{e_D})(e_G)\bigr)
  \\
 &  \stackrel{\eqref{b_d}}{=} & 
\mu(\pr(K)) \cdot b_d(\sigma(d'))\cdot \sigma(d')\left(\frac{1}{\mu(\pr(K))} \cdot r \cdot 1\right)
\\
& = &
b_d(\sigma(d'))\cdot  \sigma(d')r
\\
&  \stackrel{\eqref{b_d}}{=} & 
\begin{cases}
\frac{1}{\mu(\pr(K))} \cdot \sigma(d)r& \text{if}\; d' = d;
\\
0 & \text{otherwise}.
\end{cases}                             
 \end{eqnarray*}

 This implies for $d \in D$,  $r \in R$, and $g \in G$
 \begin{multline}
\bigl(b_d \cdot (r\cdot b_{e_D})\bigr)
\stackrel{\eqref{s_is_sum_d_s(sigma(d)_cdot_b_d}}{=}  
\sum_{d' \in D} \mu(\pr(K)) \cdot b_d \cdot (r\cdot b_{e_D}) (\sigma(d'))b_{d'}
\\
= 
\mu(\pr(K)) \cdot \left(\frac{1}{\mu(\pr(K))} \cdot \sigma(d)r\right) \cdot b_d
=
\sigma(d)r  \cdot b_d.
\label{b_d_cdot_r_cdot_b_e_is_sigma(d)_cdot_r_cdot_b_d}
\end{multline}

Recall from Lemma~\ref{unit_in_calh(G//K}
that $\calh(G//K;R,\rho,\omega)$ has a unit, namely $b_{e_D}$.

We conclude from~\eqref{b_d_1_cdot_b_d_2_is_omega_cdot_b_d_d_2}
and~\eqref{b_d_cdot_r_cdot_b_e_is_sigma(d)_cdot_r_cdot_b_d}

\begin{lemma}\label{lem:twisted_Hecke_algebra_over_K-as_crossed_product}
  Suppose that $Q$ is compact. Consider a
  compact open normal subgroup $K$ of $G$ satisfying $K \in P$.

  Then the unital ring $\calh(G//K;R,\rho,\omega)$ is the crossed
  product $R \ast D$ associated to $(w,c)$ for $w$ defined
  in~\eqref{calt(d_1,d_2)} and $c_d(r) := (\rho \circ \sigma(d))(r)$.
\end{lemma}

\subsection{Filtering the Hecke algebra of a compact group by normal compact open subgroups}%
\label{subsec:Filtering_the_Hecke_algebra_of_a_comapct_group_by_normal_compact_open_subgroups}

Consider a  sequence $G = K_0 \supseteq K_1 \supseteq K_1 \supseteq K_2 \supseteq \cdots$ of
normal compact open subgroups of $G$ with $\bigcap_{n \ge 0} K_n= \{1\}$  such that $K_n \in P$
holds for $n \in \IN$. It exists by Lemma~\ref{lem:arranging_K_to_be_normal} as we
assume throughout this section that $Q$ is compact and $N$ is locally central.  Let $1_{K_n}$ be the element in
$\calh(G;R,\rho,\omega)$ defined in~\eqref{unit_in_calh(G//K}.  Then $1_{K_n}$ is central
in $\calh(G;R,\rho,\omega)$, since  $K_n$ is normal in $G$. We have
$1_{K_n} \cdot 1_{K_m} = 1_{K_n} = 1_{K_m} \cdot 1_{K_n}$ for $m \le n$. For every
$s \in \calh(G)$ there exists a natural number $n \in \IN$ satisfying
$1_{K_n} \cdot s = s = s \cdot 1_{K_n}$. In the sequel we sometimes abbreviate $1_n = 1_{K_n}$, 
$\calh(G) = \calh(G;R,\rho,\omega)$ and $\calh(G//K_n) = \calh(G//K_n;R,\rho,\omega)$ and put $1_{-1} = 0$.
The elementary proof of the next lemma is left to the reader.

\begin{lemma}\label{lem:decomposing_calh(G;R,rho,omega)}
  We have the subrings rings $1_n\calh(G)1_n = \calh(G//K_n)$  and $(1_n - 1_{n-1})\calh(G)(1_n - 1_{n-1})$
  of $\calh(G)$, which have $1_n$ and $(1_n - 1_{n-1})$ as unit. We get an obvious
  identification of rings (without unit)
  \[
    \bigoplus_{m \ge 0} (1_m - 1_{m-1})\calh(G)(1_m - 1_{m-1}) = \calh(G),
   \]
   and for $n \ge 0$ of rings with unit
    \[
     \bigoplus_{m  = 0}^n   (1_m - 1_{m-1})\calh(G)(1_m - 1_{m-1}) = 1_n\calh(G)1_n.
     \]
 \end{lemma}

  Recall that a sequence $A_0 \xrightarrow{f_0} A_1 \xrightarrow{f_1} A_2$ in an additive
category $\cala$ is called \emph{exact at $A_1$}, if $f_1 \circ f_0 = 0$ and for every
object $A$ and morphism $g \colon A \to A_1$ with $f_1 \circ g = 0$ there exists a
morphism $\overline{g} \colon A \to A_0$ with $f_0 \circ \overline{g} = g$.  For
information how this notion is related by the Yoneda embedding to the usually notion of
exactness for modules we refer to~\cite[Lemma~5.10 and Lemma~6.3]{Bartels-Lueck(2020additive)}.
A functor $F \colon \cala \to \cala'$ of additive categories is called \emph{faithfully
  flat}, provided that a sequence $A_0 \xrightarrow{f_0} A_1 \xrightarrow{f_1} A_2$ in
$\cala$ is exact, if and only if the sequence
$F(A_0) \xrightarrow{F(f_0)} F(A_1) \xrightarrow{F(f_1)} F(A_2)$ in $\cala'$ is exact.

\begin{lemma}\label{lem:inclusion_of_a_factor} Let $S$ and $T$ be unital rings. Let
  $\pr \colon S \times T \to T$ be the projection, which is a homomorphism of unital rings.
  Let $i \colon S \to S \times T$ be the inclusion sending $s$ to $(s,0)$, which is a
  homomorphism of rings (without units).  Then
  \begin{enumerate}
    \item\label{lem:Iem:inclusion_of_a_factor:Idem(i_oplus)_and_res_pr} 
There exists a diagram of unital additive categories commuting up to natural equivalence of
  unital additive categories
\[
\xymatrix@!C=8em{\Idem(\underline{S}_{\oplus}) \ar[r]^-{\Idem(i_{\oplus})} \ar[d]_{\Theta_S}^{\simeq} 
&
\Idem(\underline{S \times T}_{\oplus}) \ar[d]^{\Theta_{S \times T}}_{\simeq} 
\\
\MODcat{S}_{\fgp} \ar[r]_-{\pr^*}
& \MODcat{S \times T}_{\fgp}
}
\]
where the vertical arrows are the equivalences of unital additive categories of~\eqref{underline(r)_oplus_to_MODcat(R)_fgp}
and $\pr^*$ is restriction with $\pr$;

\item\label{em:inclusion_of_a_factor:retraction} 
  The functor
$\Idem(\underline{i}_{\oplus}) \colon \Idem(\underline{S}_{\oplus}) \to \Idem(\underline{S \times T}_{\oplus})$
has retraction, namely
$\Idem(\underline{\pr}_{\oplus}) \colon \Idem(\underline{S\times T }_{\oplus}) 
\to 
\Idem(\underline{S }_{\oplus});
$

\item\label{em:inclusion_of_a_factor:exactness} 
The functor
$\Idem(\underline{i}_{\oplus}) \colon \Idem(\underline{S}_{\oplus}) 
\to 
\Idem(\underline{S \times T}_{\oplus})$
is faithfully flat.
\end{enumerate}
\end{lemma}
\begin{proof}~\ref{lem:Iem:inclusion_of_a_factor:Idem(i_oplus)_and_res_pr}
Next we construct for every object $([l],p)$ in $\Idem(S_{\oplus})$ an isomorphism in
$\MODcat{S \times T}_{\fgp}$
\[
  T([l],p) \colon  \pr^*\circ \;\Theta_S([l],p)
  \xrightarrow{\cong} \Theta_{S \times T} \circ \Idem(i_{\oplus})([l],p).
\]
Let $A$ be the $(l,l)$-matrix over $S$, for which $p \colon [l] \to [l] $ is given by
$A$. If $i(A)$ is the $(l,l)$-matrix over $S \times T$ given by applying $i$ to each
element in $A$, then $\theta_{S \times T} \circ i_{\oplus}(p)$ is the
$S \times T$-homomorphism $r_{i(A)} \colon (S \times T)^l \to (S \times T)^{l}$ given by
right multiplication with $i(A)$.  Let $i^l \colon S^l \to (S \times T)^l$ be the map
sending $(x_1, x_2, \ldots, x_l)$ to $(i(x_1), i(x_2), \ldots, i(x_l))$. We obtain a
commutative diagram of abelian groups
\[\xymatrix{S^l \ar[r]^-{i^l} \ar[d]_{r_A} & (S \times T)^l \ar[d]^{r_{i(A)}}
    \\
    S^{l } \ar[r]_-{i^{l'}} & (S \times T)^{l}.  }
\]
Now $i^{l'}$ induce a homomorphism of abelian groups.
\[
T([l],p) \colon \im(r_A) \to \im(r_{i(A)}).
\]
It is injective, since $i$ and hence $i^{l}$ is injective. Next we show that $T([l],p)$ is
bijective.  Let $y$ be an element of the image of $r_{i(A)}$.  Choose
$x = \bigl((s_1,t_1), \ldots, (s_l,t_l)\bigr)$ in $(S \times T)^l$ with $r_{i(A)}(x) = y$.
Define $x' \in S$ by $x' = (s_1, \ldots, s_l)$. Then
$r_{i(A)} \circ i^l(x') = r_{i(A)}(x) = y$. Hence $i^{l}$ sends $r_A(x)$ to $y$. This
finishes the proof that $T([l],p)$ is an isomorphisms of abelian groups.  One easily
checks that it is an isomorphism of $S \times T$-modules.

We leave it to the reader to check that the collection of the isomorphisms $T([l],p)$
defines a natural equivalence of functors
$\Idem(\underline{S}_{\oplus}) \to \MODcat{S \times T}_{\fgp}$ from $\pr^* \circ \; \theta_S$
to $\theta_{S \times T} \circ \Idem(i_{\oplus})$.
\\[1mm]~\ref{em:inclusion_of_a_factor:retraction}
This follows from $\pr \circ i = \id_{S}$.
\\[1mm]~\ref{em:inclusion_of_a_factor:exactness}
Since restriction is faithfully flat, the claim follows from
assertion~\ref{lem:Iem:inclusion_of_a_factor:Idem(i_oplus)_and_res_pr}.
\end{proof}

We record for later purposes

\begin{lemma}\label{lem:criterion_for_flatness_for_Hecke_algebras}
  Suppose that $Q$ is compact. Consider normal compact open subgroups
  $K$ and $K'$ of $G$ satisfying $K' \subseteq K$ and $K,K' \in P$. Let
 \[
 i \colon \calh(G//K;R,\rho,\omega) \to \calh(G//K';R,\rho,\omega )
\]
 be  the  inclusion of rings. Let $m \ge 0$ be an integer. Denote by
 \[
 i[\IZ^m]  \colon \calh(G//K;R,\rho,\omega)[\IZ^m] \to \calh(G//K';R,\rho, \omega)[\IZ^m]
\]
the  inclusion of the (untwisted) group rings induced by $i$.

Then the functor
\begin{multline*}
  \Idem(\underline{i[\IZ^m]}_{\oplus}) \colon
  \Idem\bigl(\underline{\calh(G//K;R,\rho, \omega)[\IZ^m]}_{\oplus}\bigr) 
  \\
 \to 
\Idem\bigl(\underline{\calh(G//K';R,\rho, \omega)[\IZ^m]}_{\oplus}\bigr)
\end{multline*}
has  a retraction and is faithfully flat.
\end{lemma}
\begin{proof} This follows from  Lemma~\ref{lem:inclusion_of_a_factor} and the 
  the decomposition of unital rings 
  \[
  \calh(G//K';R,\rho, \omega)  = \calh(G//K;R,\rho, \omega) \oplus (1_{K'} - 1_K)\calh(G//K';R,\rho, \omega)(1_{K'} - 1_K),
  \]
    cf. Lemma~\ref{lem:decomposing_calh(G;R,rho,omega)}.
\end{proof}

\subsection{Proof of Theorem~\ref{the:regular_coherence_of_the_Hecke_algebra}}%
\label{subsec:Proof_of_Theorem_ref(the:regular_coherence_of_the_Hecke_algebra)}

\begin{lemma}\label{regularity_and_direct-sums}
  Let $\cala_i$ be a collection of additive categories.
  Then $\bigoplus_{i \in I} \cala_i$ is  $l$-uniformly regular
   or regular  respectively, if and only if
  each $\cala_i$ is  $l$-uniformly regular  or regular  respectively.
\end{lemma}
\begin{proof}
  This is a consequence of  the observations following from~\cite[Lemma~5.3]{Bartels-Lueck(2020additive)},
  that for an object
  $A \in \bigoplus_{i \in I} \cala_i$ there exists a finite subset $J \subseteq I$
  with $A \in \bigoplus_{i \in J} \cala_i$ and we have the identifications
  \begin{eqnarray*}
  \mor_{\bigoplus_{i \in I}\cala_i}(?,A)
   & = &
   \mor_{\bigoplus_{i \in J}\cala_i}(?,A);
  \\
    \MODcat{\IZ(\bigoplus_{i \in J} \cala_i)}
    & = &
  \prod_{i \in J} \MODcat{\IZ\cala_i}.
  \end{eqnarray*}
  More details of the proof  can be found in~\cite[Section~11]{Bartels-Lueck(2020additive)}.
\end{proof}

Consider a sequence $G = K_0 \supseteq K_1 \supseteq K_1 \supseteq K_2 \supseteq \cdots$ of
normal compact open subgroups of $G$ with $\bigcap_{n \ge 0} K_n= \{1\}$ such that $K_n \in P$
 holds for $n \in \IN$.  We get from Lemma~\ref{lem:decomposing_calh(G;R,rho,omega)}
identifications  of additive categories
\begin{eqnarray*}
  \bigoplus_{m \ge 0} \Idem\bigl(\underline{(1_m - 1_{m-1})\calh(G)(1_m - 1_{m-1})}_{\oplus}\bigr)[\IZ^m]
   & = &
   \Idem\bigl(\underline{\calh(G)}_{\oplus}\bigr)[\IZ^m];
  \\
 \bigoplus_{m =  0}^n  \Idem\bigl(\underline{(1_m - 1_{m-1})\calh(G)(1_m - 1_{m-1})}_{\oplus}\bigr)[\IZ^m]
  & = &
 \Idem\bigl(\underline{\calh(G/K_n)}_{\oplus}\bigr)[\IZ^m].
\end{eqnarray*}
Hence by   Lemma~\ref{regularity_and_direct-sums}
it suffices to show that $\Idem\bigl(\underline{\calh(G/K_n)}_{\oplus}\bigr)[\IZ^m]$
is $(l+2m)$-uniformly regular or regular respectively for every $n \in \IN$.

The unital ring $\calh(G//K_n)$ is $l$-uniformly regular or regular
respectively, since $R$ is $l$-uniformly regular or regular
respectively by assumption and we have
Lemma~\ref{lem:regular_passes_to_finite_crossed_product_rings}
and Lemma~\ref{lem:twisted_Hecke_algebra_over_K-as_crossed_product}.
Hence $\Idem\bigl(\calh(G//K_n)\bigr)[\IZ^m]$ is $(l+2m)$-uniformly
regular or regular respectively by~\cite[Corollary~6.5 and
Theorem~10.1]{Bartels-Lueck(2020additive)}.  This finishes the proof
of Theorem~\ref{the:regular_coherence_of_the_Hecke_algebra}.


\typeout{---- Section 8:  Negative $K$-groups and the projective class group of Hecke algebras for compact td-groups-------}

\section{Negative $K$-groups and the projective class group of Hecke algebras over compact td-groups}%
\label{sec:Negative_K-groups_and_the_projective_class_group_of_Hecke_algebras_over_compact_td_groups}

Let $G$, $N$, $Q := G/N$, $\pr \colon G \to Q$, $R$, $\calp$, $\rho$, $\omega$,
and $\mu$ be as in
Subsection~\ref{subsec:normal_characters} and denote by
$\calh(G;R,\rho, \omega)$ the Hecke algebra, which we have introduced in
Subsection~\ref{subsec:The_construction_of_the_Hecke_algebra}. Our main assumption
in this section will be that $Q$ is compact.

\begin{lemma}\label{lem:negative_K-groups_and_K_0_of_calh(Q)_for_compact_Q}
  Suppose that $Q$ is compact and $N$ is locally central. Suppose that the unital ring $R$ is  regular
  and satisfies $\IQ \subseteq R$.  Then:

  \begin{enumerate}
  
  \item\label{lem:negative_K-groups_and_K_0_of_calh(Q)_for_compact_Q:K_n_is_colimit}
    Let $\calk$ be the set of compact open
    normal subgroups $K \subseteq G$ with $K \in P$  directed by
    $K \le K' \Longleftrightarrow K' \subseteq K$.

   Then we get for $n \in \IZ$
   \[
     K_n\bigl(\calh(G;R,\rho,\omega)\bigr)
     = \colim_{K \in \calk}  K_n\bigl(\calh(G//K;R,\rho,\omega)\bigr);
   \]

  \item\label{lem:negative_K-groups_and_K_0_of_calh(Q)_for_compact_Q:negative_K_groups}
   We get 
  \[
   K_n\bigl(\calh(G;R,\rho,\omega)\bigr)  = 0 \quad \text{for}\; n \le -1.
   \]

\end{enumerate}

\end{lemma}
\begin{proof}~\ref{lem:negative_K-groups_and_K_0_of_calh(Q)_for_compact_Q:K_n_is_colimit}
  We conclude from Lemma~\ref{lem:calh(H)_as_union}  and Lemma~\ref{lem:arranging_K_to_be_normal}
     \[
   \calh(G;R,\rho,\omega) = \bigcup_{K \in \calk}  \calh(G//K;R,\rho,\omega).
   \]
   Now apply~\eqref{colim_K_n(cala_i)_to_k_n(cala)}.
   \\[1mm]~\ref{lem:negative_K-groups_and_K_0_of_calh(Q)_for_compact_Q:negative_K_groups}
   For $K \in \calk$ the unital ring 
   $\calh(G//K;R,\rho,\omega) $ is regular by
   Lemma~\ref{lem:regular_passes_to_finite_crossed_product_rings}~%
\ref{lem:regular_passes_to_finite_crossed_product_rings:regular}
   and Lemma~\ref{lem:twisted_Hecke_algebra_over_K-as_crossed_product}.  Hence 
   $ K_n\bigl(\calh(G//K;R,\rho, \omega)\bigr) = \{0\}$ for $n \le -1$, see~\cite[page~154]{Rosenberg(1994)}.
   Now apply
   assertion~\ref{lem:negative_K-groups_and_K_0_of_calh(Q)_for_compact_Q:K_n_is_colimit}.
 \end{proof}

\begin{remark}\label{erem:nested_sequence_and-colimit_Q-compact}
  Suppose that $Q$ is compact and $N$ is locally central. Because of
  Lemma~\ref{lem:arranging_K_to_be_normal} we can choose a nested sequence of elements in
  $\calk$
\[
K_0 \supseteq K_1 \supseteq K_2 \supseteq K_3 \supseteq \cdots
\]
satisfying $\bigcap_{i = 0}^{\infty} K_n = \{1\}$. Then for every $K \in \calk$ there is a
natural number $i$ with $K_i \subseteq K$.  Abbreviate
$\calh(G//K_i) = \calh(G//K_i;R,\rho,\omega)$. Then the inclusion
$\calh(G//K_{i}) \to \calh(G//K_{i+1})$ induces a split injection
$K_n(\calh(G//K_i)) \to K_n(\calh(G//K_{i+1}))$ for $i \in \IN$ and $n \in \IZ$ by
Lemma~\ref{lem:criterion_for_flatness_for_Hecke_algebras}.
Lemma~\ref{lem:negative_K-groups_and_K_0_of_calh(Q)_for_compact_Q}~%
\ref{lem:negative_K-groups_and_K_0_of_calh(Q)_for_compact_Q:K_n_is_colimit} implies that
there is an isomorphism
 \begin{multline*}
   K_n(\calh(G;R,\rho,\omega) )
   \\
   \cong
   K_n(\calh(G//K_0)) \oplus \bigoplus_{i \ge 0}  \cok\bigl(K_n(\calh(G//K_i)) \to K_n(\calh(G//K_{i+1}))\bigr)
 \end{multline*}
 and $\cok\bigl(K_n(\calh(G//K_i)) \to K_n(\calh(G//K_{i+1}))\bigr)$ is isomorphic to a direct summand of
 $K_n(\calh(G//K_{i+1}))$.

 Now suppose additionally that $R$ is semisimple. Then $\calh(G//K_i)$ is semisimple and hence
 the abelian group $K_0(\calh(G//K_i))$ is finitely generated free 
 for $i \in \IN$ by Lemma~\ref{lem:regular_passes_to_finite_crossed_product_rings}~%
\ref{lem:regular_passes_to_finite_crossed_product_rings:semi-simple} and
 Lemma~\ref{lem:twisted_Hecke_algebra_over_K-as_crossed_product}, Hence
 the abelian group $K_0(\calh(G;R,\rho,\omega))$ is free and in particular torsionfree.
\end{remark}


\typeout{---- Section 9:  On the algebraic $K$-theory of the Hecke algebra of covirtually $\IZ$ totally disconnected groups-------}

\section{On the  algebraic $K$-theory of the Hecke algebra of a covirtually $\IZ$ totally disconnected group}%
\label{sec:On_the_algebraic_K-theory_of_the_Hecke_algebra_of_covirtually_Z_totally_disconnected_group}

Consider the setup of Section~\ref{subsec:covirtually_Z_groups}. In particular $Q$ is covirtually cyclic.
Denote by $\bfT_{\bfKinfty(\phi^{-1})}$ the mapping torus
of the map 
\[
  \bfKinfty(\phi^{-1}) \colon \bfKinfty(\calh(L;R,\rho|_L,\omega))
  \to \bfKinfty(\calh(L;R,\rho|_L,\omega))
\]
of non-connective $K$-theory spectra.

\begin{theorem}[Wang sequence]\label{the:Wang_sequence}
  Suppose that the unital ring $R$ is  regular
  and satisfies $\IQ \subseteq R$. Assume that $N$ is locally central.
Then:
\begin{enumerate}

\item\label{the:Wang_sequence:weak_homotopy_equivalence}

There is a weak homotopy equivalence of non-connective spectra
\[
  \bfa^{\infty}  \colon \bfT_{\bfKinfty(\phi^{-1})}
  \xrightarrow{\simeq}  \bfKinfty(\calh(G;R,\rho,\omega));
\]

\item\label{the:Wang_sequence:Wang_sequence}
We get a long exact sequence, infinite to the left
\begin{multline*}
  \cdots \xrightarrow{K_2(i)}  K_2(\calh(G;R,\rho,\omega))
  \xrightarrow{\partial_2}   K_1(\calh(L;R,\rho|_L,\omega)) 
\\
\xrightarrow{\id - K_1(\phi^{-1})}  K_1(\calh(L;R,\rho|_L, \omega)) 
\xrightarrow{K_1(i)}  K_1(\calh(G;R,\rho, \omega)) 
\\
\xrightarrow{\partial_1}  K_0(\calh(L;R,\rho|_L,\omega)) 
\xrightarrow{\id - K_0(\phi^{-1})}   K_0(\calh(L;R,\rho|_L,\omega)) 
\\
\xrightarrow{K_0(i)}  K_0(\calh(G;R,\rho, \omega)) \to 0;
\end{multline*}

\item\label{the:Wang_sequence:negative_K-groups}
We get for $n \le 1$
\[
K_n(\calh(G;R,\rho, \omega)) = 0.
\]

\end{enumerate}
\end{theorem}
\begin{proof}~\ref{the:Wang_sequence:weak_homotopy_equivalence} This follows from
  Lemma~\ref{lem:K-theory_of_virt_Z_reduced_to_Laurent} and
  Theorem~\ref{the:The_non_connective_K-theory_of_additive_categories} applied to the
  additive category
  $\cala = \Idem\bigl(\underline{\calh(L;R,\rho|_L,\omega)}_{\oplus}\bigr)$ after we have shown that
  the  additive category $\Idem\bigl(\underline{\calh(L;R,\rho|_L,\omega)}_{\oplus}\bigr)$
  is regular.  This has already been done in 
  Theorem~\ref{the:regular_coherence_of_the_Hecke_algebra}.
  \\[1mm]~\ref{the:Wang_sequence:Wang_sequence}
  and~\ref{the:Wang_sequence:negative_K-groups} These follow from the Wang sequence
  associated to the left hand side of the weak homotopy equivalence appearing in
  assertion~\ref{the:Wang_sequence:weak_homotopy_equivalence} and
  Lemma~\ref{lem:negative_K-groups_and_K_0_of_calh(Q)_for_compact_Q}~%
\ref{lem:negative_K-groups_and_K_0_of_calh(Q)_for_compact_Q:negative_K_groups}.
\end{proof}


\typeout{---- Section 10:  Some input  for  the Farrell-Jones Conjecture -------}

\section{Some input for the Farrell-Jones Conjecture}%
\label{sec:Some_input_for_the_Farrell-Jones_Conjecture}

In forthcoming papers we will need for the proof and the application of the $K$-theoretic
Farrell-Jones Conjecture for the Hecke algebra of a closed subgroup of a reductive $p$-adic
group, which is our ultimate goal,
Theorem~\ref{the:regular_coherence_of_the_Hecke_algebra} and the following
Theorem~\ref{the:Q'_to_Q-faithfully_flat}.

Consider the setup of
Subsection~\ref{subsec:normal_characters}.  For the
remainder of this subsection we will assume that the td-group $Q$ is compact and $N$ is
  locally central.  Let $\overline{i} \colon Q' \to Q$ be the inclusion of a compact open
  subgroup of $Q$. Put $G' = \pr^{-1}(Q')$. Let $i \colon G' \to G$ be the inclusion. The
  construction in Subsection~\ref{subsec:Functoriality_in_Q} yields a ring homomorphism
\[
  \calh(i) \colon \calh(G';R,\rho',\omega)\to \calh(G;R,\rho,\omega)
\]
where $\rho' = \rho \circ i$,  $\mu'$ is  obtained from $\mu$ by restriction with $i$, and we take $N' = N$
and $\omega' = \omega$.  The image
$\calh(i)(s)$ of an element $s \in \calh(G';R,\rho',\omega)$, which is given by an
appropriate function $s \colon G' \to R$, is specified by the function
$\calh(i)(s) \colon G \to R$ sending $g$ to $s(g)$,  if $g \in G'$, and to $0$, if
$g \notin G'$, see
Lemma~\ref{lem:properties_of_phi_ast(s')}~\ref{lem:properties_of_phi_ast(s'):injective_phi}.

\begin{theorem}\label{the:Q'_to_Q-faithfully_flat}
  Suppose that $Q$ is compact and $N$ is locally central.
  Then the functor  of unital additive categories
\[
\Idem\bigl(\underline{\calh(i)}_{\oplus}[\IZ^m]\bigr) 
\colon  
\Idem\bigl(\underline{\calh(G';R,\rho', \omega)}_{\oplus}[\IZ^m]\bigr) 
\to 
\Idem\bigl(\underline{\calh(G;R,\rho,\omega)}_{\oplus}[\IZ^m]\bigr)
\]
is faithfully flat.
\end{theorem}
\begin{proof}
  Let $\calk'$ be the directed set of normal compact open subgroups of $Q$ which satisfy
  $K \subseteq Q'$, and $K \in P$, where we put
  $K \le K' \Longleftrightarrow K' \subseteq K$.  Note that for any compact open subgroup
  $L$ of $Q$ there exists $K \in \calk'$ with $K \subseteq L$ by
  Lemma~\ref{lem:arranging_K_to_be_normal}.

In the sequel we abbreviate
\begin{eqnarray*}
  \calh(G)
  & := &
 \calh(G;R,\rho, \omega);
  \\
  \calh(G//K)
  & := &
 \calh(G//K;R,\rho, \omega),
\end{eqnarray*}
and analogously for $G'$. Next we want to show that the functor
\[
\Idem\bigl(\underline{j_K}_{\oplus}[\IZ^m]\bigr) 
\colon  
\Idem\bigl(\underline{\calh(G//K)}_{\oplus}[\IZ^m]\bigr) 
\to 
\Idem\bigl(\underline{\calh(G)}_{\oplus}[\IZ^m]\bigr)
\]
is faithfully flat for $K \in \calk'$, where $i_K \colon \calh(G//K) \to \calh(G)$ is the
inclusion.  Consider morphisms $f_0 \colon A_0 \to A_1$ and $f_1 \colon A_1 \to A_2$ in
$\Idem\bigl(\underline{\calh(G//K)}_{\oplus}[\IZ^m]\bigr)$ with $f_1 \circ f_0 = 0$.
Note that we can consider them also
as morphisms in $\Idem\bigl(\underline{\calh(G)}_{\oplus}[\IZ^m]\bigr)$. We have to show
that it is exact in $\Idem\bigl(\underline{\calh(G//K)}_{\oplus}[\IZ^m]\bigr)$, if and only
if it is exact in $\Idem\bigl(\underline{\calh(G)}_{\oplus}[\IZ^m]\bigr)$.

Suppose that $A_0 \xrightarrow{f_0} A_1 \xrightarrow{f_2} A_2$ is exact in
$\Idem\bigl(\underline{\calh(G//K)}_{\oplus}[\IZ^m]\bigr)$.  In order to show that it is
exact in $\Idem\bigl(\underline{\calh(G)}_{\oplus}[\IZ^m]\bigr)$, we have to find for any
object $A$ and any morphism $g \colon Q \to A_1$ in
$\Idem\bigl(\underline{\calh(G)}_{\oplus}[\IZ^m]\bigr)$ with $f_1 \circ g = 0$ a morphism
$\overline{g}\colon A \to A_0$ in $\Idem\bigl(\underline{\calh(G)}_{\oplus}[\IZ^m]\bigr)$
with $f_0 \circ \overline{g} = g$. We can choose an element $K' \in \calk'$ with $K \le K'$
such that $A$ and $g$ live already in
$\Idem\bigl(\underline{\calh(G//K')}_{\oplus}[\IZ^m]\bigr)$ by
Lemma~\ref{lem:calh(H)_as_union}. Since the inclusion
\[\Idem\bigl(\underline{\calh(G//K)}_{\oplus}[\IZ^m]\bigr) \to
\Idem\bigl(\underline{\calh(G//K')}_{\oplus}[\IZ^m]\bigr)
\]
is faithfully flat by Lemma~\ref{lem:criterion_for_flatness_for_Hecke_algebras}, we can
find $\overline{g} \colon A \to P_0$ with $f_0 \circ \overline{g} = g$ in
$\Idem\bigl(\underline{\calh(G//K')}_{\oplus}[\IZ^m]\bigr)$ and hence also in
$\Idem\bigl(\underline{\calh(G)}_{\oplus}[\IZ^m]\bigr)$.

Suppose that $A_0 \xrightarrow{f_0} A_1 \xrightarrow{f_2} A_2$ is exact in
$\Idem\bigl(\underline{\calh(G)}_{\oplus}[\IZ^m]\bigr)$.  In order to show that it is
exact in $\Idem\bigl(\underline{\calh(G//K)}_{\oplus}[\IZ^m]\bigr)$ we have to find for
any object $A$ and any morphism $g \colon Q \to A_1$ in
$\Idem\bigl(\underline{\calh(G//K)}_{\oplus}[\IZ^m]\bigr)$ with $f_1 \circ g = 0$ a
morphism $\overline{g}\colon A \to A_0$ in
$\Idem\bigl(\underline{\calh(G//K)}_{\oplus}[\IZ^m]\bigr)$ with
$f_0 \circ \overline{g} = g$.  At any rate we can find such $\overline{g} \colon A \to A_1$ in
$\Idem\bigl(\underline{\calh(G)}_{\oplus}[\IZ^m]\bigr)$. We conclude from
Lemma~\ref{lem:calh(H)_as_union} that there exists $K' \in \calk'$ with $K \le K'$ such that
$\overline{g} \colon A \to A_1$ lies already in $\Idem\bigl(\underline{\calh(G//K')}_{\oplus}[\IZ^m]\bigr)$.
Recall from Lemma~\ref{lem:criterion_for_flatness_for_Hecke_algebras}
that there is a retraction  of the inclusion
\[\Idem\bigl(\underline{\calh(G//K)}_{\oplus}[\IZ^m]\bigr) \to
\Idem\bigl(\underline{\calh(G//K')}_{\oplus}[\IZ^m]\bigr)
\]
If we apply it to $\overline{g}$, we get a morphism
$\overline{g'} \colon A \to A_1$ in $\Idem\bigl(\underline{\calh(G//K)}_{\oplus}[\IZ^m]\bigr)$
satisfying $f_1 \circ \overline{g}' = g$ in $\Idem\bigl(\underline{\calh(G//K)}_{\oplus}[\IZ^m]\bigr)$.
This finishes the proof that  functor $\Idem\bigl(\underline{j_K}_{\oplus}[\IZ^m]\bigr)$
is faithfully flat. Analogously one shows that the functor
$
\Idem\bigl(\underline{j'_K}_{\oplus}[\IZ^m]\bigr) 
\colon  
\Idem\bigl(\underline{\calh(G'//K)}_{\oplus}[\IZ^m]\bigr) 
\to 
\Idem\bigl(\underline{\calh(G')}_{\oplus}[\IZ^m]\bigr)
$
is faithfully flat for the inclusion $j_K' \colon \calh(G'//K) \to \calh(G')$.

We have  the following commutative diagram of functors of additive categories
\[
\xymatrix@!C=17em{\Idem\bigl(\underline{\calh(G')}_{\oplus}[\IZ^m]\bigr)
\ar[r]^-{\Idem(\underline{\calh(i)}_{\oplus}[\IZ^m])}
&
\Idem\bigl(\underline{\calh(G)}_{\oplus}[\IZ^m]\bigr) 
\\
\Idem\bigl(\underline{\calh(G'//K)}_{\oplus}[\IZ^m]\bigr)
\ar[r]_-{\Idem(\underline{\calh(i//K)}_{\oplus}[\IZ^m])}
\ar[u]^{\Idem(\underline{j'_K}_{\oplus}[\IZ^m])}
&
\Idem\bigl(\underline{\calh(G//K)}_{\oplus}[\IZ^m]\bigr)
\ar[u]_{\Idem(\underline{j'_K}_{\oplus}[\IZ^m])}
}
\]
whose two left vertical arrows are faithfully flat. We conclude from
Lemma~\ref{lem:calh(H)_as_union} that it suffices to show
that the lower vertical arrow in the diagram above is faithfully flat.

We have identified $\calh(G//K)$ and $\calh(G'//K)$ respectively
as a crossed product ring $R \ast F$ and $R\ast F'$ respectively  for the finite group $F = G/K$
and $F' = G'/K$ respectively in Lemma~\ref{lem:twisted_Hecke_algebra_over_K-as_crossed_product}.
Moreover the inclusion $\calh(G//K)[\IZ^m] \to \calh(G'//K)[\IZ^m]$ corresponds under these identifications
to the inclusions $R \ast F[\IZ^m] \to R\ast F'[\IZ^m]$ coming from the inclusion of finite groups $F' \to F$.
The lower horizontal arrow $\Idem(\underline{\calh(i//K)}_{\oplus}[\IZ^m])$
becomes under the equivalences of categories
of~\eqref{underline(r)_oplus_to_MODcat(R)_fgp} and~\eqref{Idem_and_Z_upper_d_and_rings}
the functor
\[
  F \colon \MODcat{R \ast F'[\IZ^m]}_{\fgp} \to \MODcat{R \ast F[\IZ^m]}_{\fgp},
  \quad P \mapsto R \ast F[\IZ^m] \otimes_{R \ast F'[\IZ^m]} P.
\]
There is a commutative diagram  
\[\xymatrix{\MODcat{R \ast F'[\IZ^m]}_{\fgp} \ar[r]^-{F} \ar[d]
  &
  \MODcat{R \ast F[\IZ^m]}_{\fgp} \ar[d]
  \\
  \MODcat{R[\IZ^m]}_{\fgp} \ar[r] 
  &
  \MODcat{R[\IZ^m]}_{\fgp}
  }
\]
whose vertical arrows are given by restriction from $R \ast F[\IZ^m]$ or $R\ast F'[\IZ^m]$  to $R[\IZ^m]$
and whose lower vertical arrow is given by $P \mapsto \bigoplus_{i = 1}^{[F : F']} P$.
Since the vertical arrows and the lower horizontal arrow are obviously faithfully flat,
the upper vertical arrow is faithfully flat. This finishes the proof of
Lemma~\ref{the:Q'_to_Q-faithfully_flat}.
\end{proof}


\typeout{---- Section 11:  Characteristic $p$ -------}

\section{Characteristic $p$}\label{sec:Characteristic_p}

We have assumed $\IQ \subseteq R$, or in other words that any natural number $n \ge 1$ is
invertible in $R$.  One may wonder what happens, if one drops this condition, for
instance, if  $R$ is a field of prime characteristic.  The following  condition appearing
in~\cite[page~9]{Blondel(2011)} suffices to make sense  of the Hecke algebra.

\begin{condition}\label{con:characteristcic_p} There exists a compact open subgroup
  $K$ in $Q$ such that the index $[K:K_0]$ of any open subgroup $K_0$ of $K$ is invertible in
  $R$. 
\end{condition}

Let $Q$ be a reductive $p$-adic group. Then Condition~\ref{con:characteristcic_p} is
satisfied, if $p$ is invertible in $R$, see~\cite[page~9]{Blondel(2011)}.

However, this does not mean that assertion of the Farell-Jones Conjecture 
or Theorem~\ref{the:Wang_sequence} remains true integrally.
Our arguments would go though  if for every  compact open subgroup $K$ in $Q$
the index $[K:K_0]$ of any open subgroup $K_0$ of $K$ is invertible in $R$
which is stronger than Condition~\ref{con:characteristcic_p}.

One may hope that under under Condition~\ref{con:characteristcic_p} the Farrell-Jones
Conjecture or Theorem~\ref{the:Wang_sequence} remain true rationally.  Let us confine
ourselves to the setup of Section~\ref{subsec:covirtually_Z_groups} and
Theorem~\ref{the:Wang_sequence}.  Then we get from~\cite[Theorem~0.1]{Lueck-Steimle(2016BHS)}
a weak homotopy equivalence, where we
abbreviate $\calh(G) :=\calh(G;R,\rho,\omega)$ and analogously for $L$
\begin{multline*}
  \bfT_{\bfKinfty(\Idem(\phi)^{-1})} \vee
  \bfNKinfty(\Idem(\underline{\calh(L)}_{\oplus})_{\Idem(\underline{\phi}_{\oplus})}[t])
  \vee
  \bfNKinfty(\Idem(\underline{\calh(L)}_{\oplus})_{\Idem(\underline{\phi}_{\oplus})}[t^{-1}])
  \\
  \xrightarrow{\simeq} \bfKinfty(\calh(G;R,\rho,\omega)).
\end{multline*}
So we need to show that the homotopy groups of the Nil-terms all vanish rationally.  If
$L$ is finite, this is known to be true, see~\cite[Theorem~0.3 and
Theorem~9.4]{Lueck-Steimle(2016splitasmb)}.  Under the strong condition that there is a
sequence $L \supseteq L_1 \supseteq L_2 \supseteq L_2 \cdots$ of in $L$ normal compact open
subgroups such that $\bigcap_{i \ge 0} L_i = \{1\}$ and $\phi(L_i) = L_i$ holds for
$i \ge 0$, this implies that the homotopy groups of the Nil-terms all vanish
rationally. Without this strong condition we do not have a proof.


\typeout{----------------------------- References ------------------------------}

\addcontentsline{toc<<}{section}{References} 


\end{document}